\documentclass[11pt]{amsart}

\usepackage{amssymb,amsmath,amsthm}
\usepackage{enumerate}
\usepackage[pagebackref,colorlinks,linkcolor=red,citecolor=blue,urlcolor=blue,hypertexnames=true]{hyperref}

\DeclareMathOperator{\Ran}{Ran}
\DeclareMathOperator{\sym}{Sym}
\DeclareMathOperator{\co}{co}
\DeclareMathOperator{\Int}{int}

\renewcommand{\span}{{\rm span}}
\newcommand{\hsp}{\bullet_{\rm H}}
\renewcommand{\emptyset}{\varnothing}
\renewcommand{\setminus}{\smallsetminus}
\renewcommand{\subset}{\subseteq}

\def\De{\Delta}
\def\Ga{\Gamma}

\def\tg{g}

\def\cB{\mathcal B}

\def\Mdd{\R^{d' \times d} }
\def\CMdd{\C^{d' \times d} }
\def\Rdd{\R^{d \times d} }
\def\Cdd{{\C^{d \times d} }}
\def\SR{\mathbb S\R }
\def\mbS{\mathbb S }
\def\SRdd{\mathbb S\R^{d \times d} }
\def\SRnn{\mathbb S\R^{n \times n} }

\def\Mnns{\R^{n\times n}}
\def\Cnn{\CC^{n\times n}}

\def\T{T}

\def\bes{\begin{equation*}}
\def\ees{\end{equation*}}

\def\beq{\begin{equation} }
\def\eeq{\end{equation} }
\def\eps{\varepsilon}
\def\ep{\varepsilon}

\def\ben{\begin{enumerate} }
\def\een{\end{enumerate} }
\def\benum{\begin{enumerate} }
\def\eenum{\end{enumerate} }

\def\x{x}
\def\xs{x^*}
\def\y{y}
\newcommand{\ax}{\langle\x\rangle}
\newcommand{\axs}{\langle\x,\xs\rangle}
\newcommand{\cy}{[\y]}

\def\bmat{\left[\begin{array}{ccccccccccccccc} }
\def\emat{\end{array}\right]}

\def\bmat{\begin{bmatrix}}
\def\emat{\end{bmatrix}}
\def\beq{\begin{equation}}
\def\eeq{\end{equation}}

\def\barr{\begin{array}}
\def\earr{\end{array}}

\def\cB{\mathcal{B}}

\def\T{\ast}

\def\NN{\mathbb N}
\def\N{\mathbb N}

\def\cA{ {\mathcal A} }
\def\cB{ {\mathcal B} }
\def\cC{ {\mathcal C} }
\def\cD{ {\mathcal D} }

\def\fB{ {\mathcal B} }

\def\cH{ {\mathcal H} }

\def\cJ{ {\mathcal J}}
\def\cK{ {\mathcal K} }
\def\cL{ L }
\def\ccL{ \mathcal L }

\def\cN{ {\mathcal N} }

\def\cP{{\mathcal P}}
\def\cQ{ {\mathcal Q} }

\def\cS{{\mathcal S} }
\def\cT{{\mathcal T}}

\def\la{\lambda}

\def\tL{{\tilde L}}
\def\tcL{{\tilde\cL}}

\def\j1tog{ j= 1, \ldots, g }

\def\tA{{\widetilde A}}
\def\ta{{\tilde a}}

\def\tA{\tilde{A}}

\def\L0t{\ (L_0 \otimes I_n ) \ }

\def\PLINE1{\beta}

\newcommand{\CC}{{\mathbb C}}
\newcommand{\C}{{\mathbb C}}
\newcommand{\HH}{{\mathbb H}}
\newcommand{\K}{{\mathbb K}}
\newcommand{\RR}{{\mathbb R}}
\newcommand{\R}{{\mathbb R}}

\newcommand{\pt}{\partial}

\newtheorem{thm}{Theorem}[section]

\newtheorem{cor}[thm]{Corollary}
\newtheorem{lem}[thm]{Lemma}
\newtheorem{lemma}[thm]{Lemma}

\newtheorem{prop}[thm]{Proposition}

\theoremstyle{definition}
\newtheorem{definition}[thm]{Definition}

\theoremstyle{remark}

\newtheorem{remark}[thm]{Remark}

\numberwithin{equation}{section}

\newtheorem{exa}[thm]{Example}
\newenvironment{example}%
         {\begin{exa}}
             {{\hfill $\Box ~~$}\end{exa}}

\newcounter{Inc}

\begin{document}
\setcounter{page}{1}

\title[The matricial relaxation of an LMI]
{The matricial relaxation of a \\ linear matrix inequality}

\author[Helton]{J. William Helton${}^1$}
\address{J. William Helton, Department of Mathematics\\
  University of California \\
  San Diego}
\email{helton@math.ucsd.edu}
\thanks{${}^1$Research supported by NSF grants
DMS-0700758, DMS-0757212, and the Ford Motor Co.}

\author[Klep]{Igor Klep${}^2$}
\address{Igor Klep, Univerza v Ljubljani, Fakulteta za matematiko in fiziko \\
and
Univerza v Mariboru, Fakulteta za naravoslovje in matematiko 
}
\email{igor.klep@fmf.uni-lj.si}
\thanks{${}^2$Research supported by the Slovenian Research Agency grants 
P1-0222, P1-0288, and J1-3608.}

\author[McCullough]{Scott McCullough${}^3$}
\address{Scott McCullough, Department of Mathematics\\
  University of Florida 
   }
   \email{sam@math.ufl.edu}
\thanks{${}^3$Research supported by the NSF grant DMS-0758306.}

\subjclass[2010]{Primary 46L07, 14P10, 90C22; Secondary 11E25, 46L89, 13J30}
\date{\today}
\keywords{linear matrix inequality (LMI), completely positive,
semidefinite programming,
Positivstellensatz, Gleichstellensatz, archimedean
quadratic module,
real algebraic geometry, free positivity}

\begin{abstract}
Given linear matrix inequalities  (LMIs) $L_1$ and $L_2$ 
it is natural to ask:
\vspace{.1cm}
\ben[\rm (Q$_1$)]
\item
when does one dominate the other, that is,
does
$L_1(X)\succeq0$ imply $L_2(X)\succeq0$?
\item
when are they mutually dominant, that is, when do
 they have the same solution set? 
\een
\vspace{.1cm}

  The matrix
 cube problem of Ben-Tal and Nemirovski  \cite{B-TN}
 is an example of LMI
domination. Hence such problems can be NP-hard.
This paper describes a natural relaxation of an LMI,
based on substituting matrices for the variables $x_j$. 
With this
relaxation, the  domination questions (Q$_1$) and (Q$_2$)
have elegant answers, indeed reduce to constructible 
semidefinite programs.
As an example, to test the strength of this relaxation we specialize it
to the matrix cube problem and obtain essentially
 the relaxation given in \cite{B-TN}. 
Thus our relaxation could be viewed as generalizing it.

Assume there is an $X$ such that $L_1(X)$ and $L_2(X)$ are both positive definite, and suppose the positivity domain of $L_1$ is bounded.
For our ``matrix variable'' relaxation  
a positive answer to (Q$_1$) is equivalent to the existence
of matrices $V_j$ such that 
\beq\tag{\rm A$_1$}\label{eq:abstr}
L_2(x)=V_1^* L_1(x) V_1 + \cdots + V_\mu^* L_1(x) V_\mu.
\eeq
 As for (Q$_2$) we show that $L_1$ and $L_2$ 
  are mutually dominant if and only if, up to
 certain  redundancies described in the paper, 
 $L_1$ and $L_2$ 
 are unitarily equivalent.
Algebraic certificates for positivity,
such as \eqref{eq:abstr} for
linear polynomials,  are typically called Positivstellens\"atze.
The paper goes on to derive a  Putinar-type Positivstellensatz for 
polynomials with a cleaner and
more powerful conclusion under the stronger hypothesis
of positivity on an underlying  bounded domain of the form
$
\{X \mid L(X)\succeq0\}.
$

An observation at the core of the paper is that the relaxed LMI domination
problem is equivalent 
to a classical problem. Namely, the 
problem of determining
if a linear map $\tau$ from a subspace of matrices to a matrix algebra is
``completely positive''. Complete positivity is one of the main techniques
of modern operator theory and 
the theory of operator algebras. On one hand it provides tools for 
studying LMIs and on the other hand, since completely positive
maps are not so far from representations and generally are more
tractable than their merely positive counterparts, the theory of 
completely positive maps
provides perspective on the
difficulties in solving LMI domination problems. 
%

\end{abstract}

\maketitle

\makeatletter
\newcommand{\mycontentsbox}{%
{\centerline{NOT FOR PUBLICATION}
\addtolength{\parskip}{-2.3pt}
\small\tableofcontents}}
\def\enddoc@text{\ifx\@empty\@translators \else\@settranslators\fi
\ifx\@empty\addresses \else\@setaddresses\fi
\newpage\mycontentsbox}
\makeatother

\newpage

\section{Introduction and the statement of the main results}

In this section
we state most of our main results of the paper.
We begin with essential definitions.

\subsection{Linear pencils and LMI sets}
   For symmetric matrices $A_0,A_1,\dots,A_{g} \in \SRdd$,  the expression
\beq\label{eq:pencil}
  L(x)=A_0+\sum_{j=1}^{g} A_j x_j \in \SRdd\ax
\eeq
in noncommuting variables $x$,
  is a {\bf linear pencil}. 
If $A_0=I$, then $L$ is {\bf monic}.
If $A_0=0$, then $L$ is a {\bf truly linear pencil}.
The truly linear part $\sum_{j=1}^{g} A_j x_j$ of a linear
pencil $L$ as in \eqref{eq:pencil} will be denoted by $L^{(1)}$.

  Given
a block column matrix
  $X=\mbox{col}(X_1,\dots,X_{g}) \in (\SRnn)^{g},$
the {\bf evaluation} $L(X)$ is defined as
\beq\label{eq:trulyPencilTensor}
  L(X)=A_0\otimes I_n + \sum A_j \otimes X_j\in \mbS \R^{dn\times dn}.
\eeq
The tensor product in this expressions is the usual (Kronecker) tensor product
of matrices. We have reserved the tensor product notation for the tensor product
of matrices and have eschewed the strong temptation of using $A\otimes x_{\ell}$ in place of
$Ax_{\ell}$ when $x_{\ell}$ is one of the variables.

Let $L$ be a linear pencil.
Its {\bf matricial linear matrix inequality
(LMI) set} (also called a {\bf matricial positivity domain}) is
\beq\label{eq:posDom}
\cD_{L}:= \bigcup_{n\in\NN}\{X\in (\SRnn)^{\tg} \mid L(X) \succeq 0\}.
\eeq
Let
\begin{eqnarray}
\cD_L(n)&=&\{ X\in (\SRnn)^{\tg}\mid L(X) \succeq 0\}=\cD_L\cap
(\SRnn)^{\tg}, \\
\pt\cD_L(n)&=&\{ X\in (\SRnn)^{\tg}\mid  L(X) \succeq 0,\,
L(X)\not\succ0\}, \\
\pt\cD_L&=&\bigcup_{n\in\NN} \pt\cD_L(n).
\end{eqnarray}
The set $\cD_L(1)\subseteq\R^g$ is the feasibility set of the semidefinite
program 
$L(X)\succeq0$ and is called a {\bf spectrahedron} by algebraic geometers.

We call $\cD_L$ {\bf bounded} if there is an $N\in\NN$ with
$\|X\|\leq N$ for all $X\in \cD_L$.
We shall see later below (Proposition \ref{prop:boundedn}) that
$\cD_L$ is bounded if and only if $\cD_L(1)$ is bounded.

\subsection{Main results on LMIs}
  Here we state our main theorems giving precise algebraic
 characterizations of (matricial) LMI domination.
 While the main theme of this article is that matricial
  LMI domination problems are more tractable than their
  traditional scalar counterparts, the reader interested only
  in algorithms for the scalar setting can proceed
  to the following subsection, \S \ref{sec:algo}, 
  and then onto Section \ref{sec:comput}. 

\begin{thm}[Linear Positivstellensatz]\label{thm:intro1}
Let
$L_j\in \mbS\R^{d_j\times d_j}\ax$, $j=1,2$, 
be monic linear pencils and assume
$\cD_{\cL_1}$ is bounded.
Then $\cD_{L_1}\subseteq\cD_{L_2}$ if
and only if there is a $\mu\in\NN$ and an isometry $V\in\R^{\mu d_1\times d_2}$ such that
\beq\label{eq:linPosIntro}
L_2(x)= V^* \big( I_\mu \otimes L_1(x) \big) V.
\eeq
\end{thm}

 Suppose $L\in\mbS\R^{d\times d}\ax$,
\[
  L= I +\sum_{j=1}^g A_j x_j
\]
 is a monic linear pencil.  A subspace $\mathcal H\subseteq \mathbb R^d$
 is {\it reducing for $L$} if $\mathcal H$ reduces each
 $A_j$; i.e., if $A_j\mathcal H\subseteq \mathcal H$.  Since
 each $A_j$ is symmetric, it also follows that 
 $A_j\mathcal H^\perp \subseteq \mathcal H^\perp$.  Hence, with respect
  to the decomposition $\mathbb R^d =\mathcal H\oplus \mathcal H^\perp$,
  $L$ can be written as the direct sum,
 \[
   L = \tL \oplus \tL^\perp 
   =\begin{bmatrix} \tL & 0 \\ 0 & \tL^\perp \end{bmatrix},
\quad\text{where}\quad
  \tL= I+\sum_{j=1}^g \tA_j x_j,
\]
  and  $\tA_j$ is the restriction of $A_j$ to $\mathcal H$. 
  (The pencil $\tL^\perp$ is defined similarly.)  
  If $\mathcal H$ has dimension $\ell$, then by identifying
  $\mathcal H$ with $\mathbb R^\ell$, the pencil
  $\tL$    is a monic linear pencil of size $\ell$. 
  We say that $\tL$ is a {\it subpencil} of $L$.
  If moreover, $\cD_L=\cD_{\tL}$, then $\tL$
  is a {\it defining subpencil} and if  
  no proper subpencil of $\tL$ is defining
  subpencil for $\cD_L$, then $\tL$ is
  a {\it minimal defining $($sub$)$pencil}. 

\begin{thm}[Linear Gleichstellensatz]
 \label{thm:minimalIntro}
Let
$L_j\in \mbS\R^{d\times d}\ax$, $j=1,2$, 
be monic linear pencils with $\cD_{\cL_1}$ bounded.
Then $\cD_{\cL_1}=\cD_{\cL_2}$
if and only if 
  minimal defining pencils $\tL_1$ and $\tL_2$
for $\cD_{\cL_1}$ and $\cD_{\cL_2}$ respectively,
are unitarily equivalent. That is,
  there is a unitary  matrix
 $U$ such that 
\beq\label{eq:Intro2}
\tL_2(x)=U^* \tL_1(x) U.
\eeq 
\end{thm}

An observation at the core of these results is that the relaxed LMI domination
problem is equivalent 
to the 
problem of determining
if a linear map $\tau$ from a subspace of matrices to a matrix algebra is
{\em completely positive}.

\subsection{Algorithms for LMIs}\label{sec:algo}

Of widespread interest is determining if 
\beq\label{eq:Q1}
\cD_{L_1}(1)
\subseteq
\cD_{L_2}(1),
\eeq
or if $\cD_{L_1}(1)=\cD_{L_2}(1)$.
For example, the paper of Ben-Tal and Nemirovski 
\cite{B-TN} exhibits simple cases where determining this 
is NP-hard.
We explicitly give (in Section \ref{subsec:incl})
 a certain semidefinite program 
whose feasibility is equivalent to
$\cD_{L_1}\subseteq\cD_{L_2}$. 
Of course,
if $\cD_{L_1}\subseteq\cD_{L_2}$, then
$\cD_{L_1}(1)\subseteq\cD_{L_2}(1)$. Thus our algorithm is a
type of
relaxation of the problem \eqref{eq:Q1}.
The algorithms in this section can be read immediately
after reading Section \ref{sec:algo}.

We also have an SDP algorithm (Section \ref{subsec:bd}) easily 
adapted from the first to determine
if  $\cD_L$ is bounded, and what its ``radius'' is.
Proposition \ref{prop:boundedn} shows that $\cD_L$ is bounded
if and only if $\cD_L(1)$ is bounded.
Thus our algorithm definitively tells if $\cD_L(1)$ is a bounded set;
in addition it  yields an upper bound on the radius of $\cD_L(1)$.

In Section \ref{subsec:cube} we specialize our
relaxation to solve a matricial relaxation of the classical matrix
cube problem, finding the biggest matrix cube contained
in $\cD_L$. 
It turns out, as shown in 
Section \ref{sec:more-cube}, that our matricial relaxation
 is essentially  that of \cite{B-TN}.  
Thus the our LMI inclusion relaxation could be viewed as a
 generalization
of theirs, indeed a highly canonical one, 
in light of the precise correspondence to
classical complete positivity theory shown in \S \ref{sec:arv}. 
A potential  advantage  of our relaxation
is that there are possibilities for strengthening  it,
presented generally in Section \ref{sec:reduce} and
illustrated on the matrix cube in  Section \ref{sec:not-practical}.
 
 Finally, 
given a matricial LMI set $\cD_L$,
Section \ref{subsec:silov} gives an algorithm to compute the
linear pencil $\tL\in\mbS\Rdd\ax$ with smallest possible $d$ satisfying
$\cD_L=\cD_\tL$.

\subsection{Positivstellensatz}
Algebraic characterizations   of  polynomials $p$ which are positive
on $\cD_{L}$ are called Positivstellens\"atze 
and are classical for polynomials
on $\RR^g$.
This theory underlies
the main approach currently used for global optimization
of polynomials,
cf.~\cite{Las,Par}.
The generally noncommutative techniques in this
paper lead to 
a cleaner and more powerful commutative 
Putinar-type Positivstellensatz \cite{Put} for $p$ strictly 
positive on a bounded
spectrahedron $\cD_L(1)$.  In the theorem
which follows, $\SRdd\cy$ is the set  of symmetric
$d\times d$ matrices with entries from
$\mathbb R[y]$, the algebra of (commutative)
polynomials with coefficients from $\mathbb R$.
Note that an element of $\SRdd\cy$ may be identified
with a polynomial (in commuting variables)
with coefficients from $\SRdd$. 

\begin{thm}\label{thm:commposIntro}
Suppose $L\in\SRdd\cy$ is a monic linear pencil
and $\cD_\cL(1)$ is bounded. Then for every symmetric matrix polynomial
$p\in\R^{\ell\times\ell}\cy$ with
$p|_{\cD_\cL(1)}\succ0$, there are
$A_j\in\R^{\ell\times\ell}\cy$,
and $B_k\in\R^{d\times\ell}\cy$
satisfying
\beq\label{eq:cpos1Intro}
p=\sum_j A_j^*A_j + \sum_k B_k^* \cL B_k.
\eeq
\end{thm}

We also consider symmetric (matrices of) polynomials $p$
in noncommuting variables
with the property that $p(X)$ is positive definite
for all $X$ in a bounded matricial LMI set $\cD_L$; see Section \ref{sec:pos}.
For such noncommutative (NC) polynomials (and for even more general algebras
of polynomials, see Section \ref{sec:morePosSS})
we obtain a Positivstellensatz (Theorem \ref{thm:pos}) analogous to \eqref{eq:cpos1Intro}.
In the case that the  polynomial $p$ is linear,
this Positivstellensatz reduces to 
Theorem \ref{thm:intro1}, which can be regarded as a
``Linear Positivstellensatz''.
For perspective we mention that the proofs of our Positivstellens\"atze
actually rely on the linear Positivstellensatz.
For experts we point out that the key
reason LMI sets behave better is
that the quadratic module associated to a
monic linear pencil $L$ with bounded $\cD_L$ is archimedean.

\subsection{Outline}
The paper is organized as follows. Section \ref{sec:prelim}
collects a few basic facts about linear pencils and LMIs.
In Section \ref{sec:arv}, inclusion and equality of
matricial LMI sets are characterized and our results are then applied
in the algorithmic Section \ref{sec:comput}.
 Section \ref{sec:more-cube} gives some further details 
 about matricial relaxations of the matrix cube problem. 
The last two sections give algebraic certificates for
polynomials to be positive on LMI sets.

\section{Preliminaries on LMIs}\label{sec:prelim}

This section collects a few basic facts about linear pencils
and LMIs.

\begin{prop}\label{prop:2monic}
If $L$ is a linear pencil and $\cD_L$ contains
$0$ as 
an interior point, i.e., $0\in\cD_L\setminus\partial\cD_L$,
then there is a monic pencil $\widehat L$ with $\cD_L=\cD_{\widehat L}$.
\end{prop}

\begin{proof}
As $0\in\cD_L$, 
$L(0)=A_0$ is positive semidefinite.
Since $0\not\in\partial\cD_L$,
$A_0 \succeq \ep A_j$ for some small $\ep\in\R_{>0}$ and
all $j$.
Let
$V=$ Ran $A_0 \subseteq {\mathbb R}^d$, and set $$\widetilde A_j: =
A_j|_V \quad \text{ for }\quad  j=0, 1, \ldots, g.$$

Clearly,
$\widetilde A_0: V \to V$ is invertible and thus positive definite.
We next show that Ran $A_0$ contains Ran $A_j$ for $j\geq1$.
If $x \perp$ Ran $A_0$, i.e., $A_0 x = 0$, then
$0 = x^* A_0 x \geq \pm\, \ep x^* A_j x$ and hence
$x^* A_j x = 0$. Since $A_0 + \ep A_j \geq 0$ and
$x^* (A_0 + \ep A_j) x = 0$ it follows that
$(A_0 + \ep A_j) x = 0$, and since $A_0 x = 0$, we finally
conclude that $A_j x = 0$, i.e., $x \perp$ Ran $A_j$.
Consequently, $\widetilde A_j: V \to V$ are all symmetric and
$\cD_L = \cD_{\widetilde L}$ for
$\widetilde L=\widetilde A_0+ \sum_{j=1}^g \widetilde A_jx_j$.

To build $\widehat L$, factor $\widetilde A_0 = B^*B$ with
$B$ invertible and set
$$\widehat A_j : = B^{-*} \widetilde A_j B^{-1} \quad \text{ for }\quad  
j = 0, \ldots, g.$$
The resulting pencil $\widehat L=I+\sum_{j=1}^g \widehat A_jx_j$ is monic and 
$\cD_L = \cD_{\widehat L}$.
\end{proof}

Our primary focus will be on the matricial LMI sets 
$\cD_L$. If the spectrahedron $\cD_L(1)\subseteq\R^g$ does not
contain interior points, then (as it is a convex set) it is contained
in a proper affine subspace of $\R^g$. By reducing the number of variables
we arrive at a new pencil whose spectrahedron \emph{does} have an interior
point. By a translation we can ensure that $0$ is an interior point.
Then Proposition \ref{prop:2monic} applies and yields a monic
linear pencil with the same matricial LMI set. 
This reduction enables us to concentrate only on 
monic linear pencils in the sequel.

\begin{lem}
\label{lem:23lin}
Let $L\in\SRdd\ax$ be a linear pencil with $\cD_L$ bounded, and let
$\widehat L\in\SRnn \ax$ be another linear pencil. 
Set $s:=n (1+g)$.
Then:
\ben[\rm (1)]
\item
$\widehat L|_{\cD_L} \succ0$ if and only if $\widehat L|_{\cD_L(s)} \succ0$;
\item
$\widehat L|_{\cD_L} \succeq0$ if and only if $\widehat L|_{\cD_L(s)} \succeq0$.
\een 
\end{lem}

\begin{proof}
In both statements the direction $(\Rightarrow)$ is obvious.
If $\widehat L|_{\cD_L} \not\succ0$, there is an $\ell$, $X\in\cD_L(\ell)$ and
$v=\oplus_{j=1}^n v_j \in (\R^\ell)^n$ with 
$$\langle \widehat L(X)v,v\rangle\leq0.$$

Let 
$$
\cK:=\span \big(\{ X_i v_j \mid i=1,\ldots,g,\; j=1,\ldots,n\}\cup\{
v_j \mid j=1,\ldots,n\}\big)
.
$$
Clearly, $\dim \cK \leq s$.
Let $P$ be the orthogonal projection of $\R^{\ell}$ onto $\cK$. Then
$$
\langle \widehat L(PXP)v,v \rangle = \langle \widehat L(X)v,v \rangle \leq0.
$$
Since $PXP\in\cD_L(s)$, this proves (1). The proof of (2) is
the same.
\end{proof}

\begin{lem}\label{lem:boundedn0}
Let $L$ be a linear pencil. Then
$$\cD_L \text{ is bounded} \quad\Leftrightarrow\quad \cD_L\big((1+g)^2\big)
\text{ is bounded}.$$
\end{lem}

\begin{proof}
Given a positive $N\in\mathbb N$, 
consider the monic linear pencil
$$
\cJ_N(x)=
\frac 1N \begin{bmatrix}
N &  x_1 &  \cdots & x_g \\
x_1 & N  \\
\vdots && \ddots \\
x_g  & & & N
\end{bmatrix}
=
\frac 1N
\begin{bmatrix}
N & x^* \\
x & N I_g
\end{bmatrix}
\in
\mbS\R^{(g+1)\times(g+1)}\ax.
$$
Note that  $\cD_L$ is bounded if and only if
for some $N\in\N$, $\cJ_N|_{\cD_L}\succeq0$.
The statement of the lemma now follows from Lemma \ref{lem:23lin}.
\end{proof}

To the linear pencil $L$ we can also associate its
{\bf matricial ball} 
\bes
\cB_{L}:=\bigcup_{n\in\NN}
\{X\in (\SRnn)^{\tg}\mid \|L(X)\|\leq 1\} =
\{X\mid I- L(X)^2 \succeq 0\}.
\ees
Observe that $\cB_L=\cD_{L^\prime}$ for
\begin{equation}
  \label{lcl}
   L^\prime=\begin{bmatrix}I&L\\L&I\end{bmatrix}.
\end{equation}

\begin{prop}\label{prop:boundedn}
Let $L$ be a linear pencil. Then:
\ben[\rm (1)]
\item
$\cD_\cL$ is bounded if and only if $\cD_\cL(1)$ is bounded;
\item
$\cB_L$ is bounded if and only if  $\cB_L(1)$ is bounded.
\een
\end{prop}

\begin{proof}
(1)
The implication $(\Rightarrow)$ is obvious. For the converse
suppose $\cD_\cL$ is unbounded. By Lemma \ref{lem:boundedn0},
this means $\cD_\cL(N)$ is unbounded for some $N\in\NN$.
Then there exists a sequence $(X^{(k)})$
  from $(\SR^{N\times N})^\tg$ such
  that $\|X^{(k)}\|=1$ and a sequence
  $t_k\in\RR_{>0}$ tending to $\infty$ such that
  $\cL(t_k X^{(k)})\succeq0$.  A subsequence
   of $(X^{(k)})$ converges to $X=(X_1,\ldots,X_g)
\in (\mbS\R^{N\times N})^\tg$ which
  also has norm $1$.
For any $t$, $tX^{(k)}\to tX$ and for $k$ big enough,
$tX^{(k)}\in\cD_{\cL}$ by convexity.
So $X$ satisfies $\cL(t X)\succeq 0$ for
all $t\in\RR_{\geq0}$.

There is a nonzero vector $v$ so that $\langle X_iv,v\rangle\neq 0$ 
for at least one $i$.
Then with $Z:=(\langle X_1v,v\rangle,\ldots,
\langle X_gv,v\rangle)\in\R^g\setminus\{0\}$,
and $V$ denoting the map $V:\mathbb R\to \mathbb R^N$
 defined by $Vr=rv$, 
$$
 \cL(tZ)= (I\otimes V)^* L(tX)(I\otimes V)
$$
is nonnegative for all $t>0$, so
$\cD_\cL(1)$ is unbounded.

To conclude the proof observe that (2) is immediate from (1) using \eqref{lcl}.
\end{proof}

  A linear pencil
  $L$ is {\bf nondegenerate}, if
  it is one-one in that $L(X)=L(Y)$ implies $X=Y$ for all
  $n\in\NN$ and $X,Y\in(\SRnn)^{g}$. 
  In particular, a truly linear
pencil $L$ is nondegenerate if and only if $L(X)\neq0$ for $X\neq0$.

\begin{lemma}\label{lem:nonsingular}
  For a linear pencil $L(x)=A_0+\sum_{j=1}^{g} A_j x_j$ the following
 are equivalent:
\ben[\rm (i)]
 \item $L$ is nondegenerate;
 \item $L(Z)=L(W)$ implies $Z=W$ for all $Z,W\in\R^{g}$;
 \item the set $\{A_j \mid j=1,\ldots,g\}$ is linearly independent;
\item $L^{(1)}$ is nondegenerate.
\een
\end{lemma}

\begin{proof}
Clearly, (i) $\Leftrightarrow$ (iv). Also,
(i) $\Rightarrow$ (ii) and (ii) $\Rightarrow$ (iii) are obvious.
For the remaining implication (iii) $\Rightarrow$ (i),
assume $L(X)=L(Y)$ for some $X,Y\in(\SRnn)^{g}$. 
Equivalently, $L^{(1)}(X-Y)=0$. 
Note that $L^{(1)}(X-Y)$ equals $\sum_{j=1}^g (X_j-Y_j)\otimes A_j$ modulo
the canonical shuffle. If this expression equals $0$, then the linear
independence of the $A_1,\ldots,A_g$ (applied entrywise) implies $X=Y$.
\end{proof}

\begin{prop}\label{prop:bounded12}
Let $L=I+\sum_{j=1}^\tg A_jx_j\in\SRdd\ax$ be a monic linear pencil
 and let $L^{(1)}$ denote its truly linear part.
Then:
\begin{enumerate}[\rm (1)]
\item
$\cB_{L^{(1)}}$ is bounded if and only if $L^{(1)}$ is nondegenerate;
\item
if $\cD_\cL$ is bounded then $\{I,A_j\mid j=1,\ldots,\tg\}$
is linearly independent; the converse fails in general.
\een
\end{prop}

\begin{proof}
(1) Suppose $L^{(1)}$ is not nondegenerate, say $\sum_{j=1}^g z_j A_j=0$ for some
$z_j\in \R$. Then with $Z=(z_1,\ldots,z_\tg)\in\R^\tg$ we
have $tZ\in \cB_{L^{(1)}}$ for every $t$, so $\cB_{L^{(1)}}$ is not bounded.
Let us now prove the converse.
First, if $\cB_{L^{(1)}}$ is unbounded,  then by Proposition \ref{prop:boundedn},
  $\cB_{L^{(1)}}(1)$ is unbounded.
  So  suppose $\cB_{L^{(1)}}(1)$ is unbounded.
 Then there exists a sequence $(Z^{(k)})$
  from $\R^\tg$ such
  that $\|Z^{(k)}\|=1$ and a sequence
  $t_k\in\RR_{>0}$ tending to $\infty$ such that
  $\| L^{(1)}(t_k Z^{(k)})\|\le 1$.  A subsequence
   of $(Z^{(k)})$ converges to $Z\in\R^\tg$ which
  also has norm $1$; however, $\| L^{(1)}(Z)\|=0$
  and thus $L^{(1)}$ is degenerate.

For (2) assume
\beq\label{eq:linDep}
\la + \sum_j x_j A_j=0
\eeq
with
$\la,x_j\in\R$.
We may assume $x_j\neq 0$
for at least one index $j$.
Let $Z=(x_1,\ldots,x_\tg)\neq0$. If $\la=0$, then
$\cL(tZ)=I$ is positive semidefinite for all $t\in\RR$. Thus $\cD_\cL$
is not bounded.

Now let $\la\in\RR$ be nonzero. Then
$\cL(Z/\la)=0$. Thus, 
$\cL(t Z/\la)\succeq0$ for all $t<0$, showing $\cD_\cL$
is unbounded.

The converse of (2) fails in general.
For instance, if the $A_j$ are positive semidefinite, then
$\cD_\cL$ contains $(\RR_{\geq0})^\tg$ and thus cannot be bounded.
\end{proof}

\section{Matricial LMI sets: Inclusion and Equality}
\label{sec:arv}

Given $L_1$ and $L_2$ monic linear pencils
\beq
\label{eq:penc12}
 L_j(x) = I+\sum_{\ell=1}^\tg A_{j,\ell} x_\ell \in\mbS\R^{d_j\times d_j}\ax , \quad j=1,2,
\eeq
we shall consider the following two inclusions for
matricial LMI sets:
\beq\label{eq:bind1}
\cD_{\cL_1}\subseteq\cD_{\cL_2};\eeq
\beq\label{eq:bind2}
\partial\cD_{\cL_1}\subseteq\partial\cD_{\cL_2}.
\eeq
Equation \eqref{eq:bind1} is equivalent to:
for all $n\in\NN$ and $X\in (\SRnn)^\tg$,
$$
\cL_1(X)\succeq 0 \quad \Rightarrow \quad \cL_2(X) \succeq 0.
$$
Similarly, \eqref{eq:bind2} can be rephrased as follows:
$$
\cL_1(X)\succeq 0\quad \mathrm{and} \quad  \cL_1(X)\not\succ0 \quad \Rightarrow \quad \cL_2(X)\succeq 0\quad \mathrm{and} \quad
  \cL_2(X)\not\succ0.
$$

In this section we characterize precisely the relationship between
$L_1$ and $L_2$ satisfying \eqref{eq:bind1} and \eqref{eq:bind2}.
Section \ref{sec:tau} handles \eqref{eq:bind1} and gives
a Positivstellensatz for linear pencils. Section \ref{subsec:equalD}
shows that ``minimal'' pencils $L_1$ and $L_2$ satisfying \eqref{eq:bind2}
are the same up to unitary equivalence.

\begin{example}\label{ex:nonScalar}
By Lemma \ref{lem:23lin} it is enough to
 test condition \eqref{eq:bind1}
on matrices of some fixed (large enough) size. It is, however, not enough
to test on $X\in\R^\tg$. For instance, let
$$
\De(x_1,x_2)=
I+ 
 \begin{bmatrix}  0 & 1 & 0 \\
                     1 & 0 & 0\\
                     0 & 0 & 0
                      \end{bmatrix} x_1 +
\begin{bmatrix}  0 & 0 & 1 \\
                     0 & 0 & 0\\
                     1 & 0 & 0
                      \end{bmatrix} x_2 =
\begin{bmatrix}  1 & x_1 & x_2 \\
                     x_1 & 1 & 0\\
                     x_2 & 0 & 1
                      \end{bmatrix}\in\mbS\R^{3\times 3}\ax
$$
and
$$
\Ga(x_1,x_2)=I+ \begin{bmatrix}1  & 0 \\ 0 & -1
\end{bmatrix} x_1 +  \begin{bmatrix} 0  & 1 \\ 1 & 0
\end{bmatrix} x_2 =  \begin{bmatrix} 1+x_1  & x_2 \\ x_2 & 1-x_1
\end{bmatrix}\in\mbS\R^{2\times 2}\ax.
$$Then
\[
\begin{split}
\cD_{\De} & =\{(X_1,X_2)\mid 1-X_1^2-X_2^2 \succeq0\}, \\
\cD_{\De}(1) & =\{(X_1,X_2)\in\R^2\mid X_1^2+X_2^2\leq 1\}, \\
\cD_{\Ga}(1) & =\{(X_1,X_2)\in\R^2\mid X_1^2+X_2^2\leq 1\}.
\end{split}
\]
Thus $\cD_{\De}(1)=\cD_{\Ga}(1)$. On one hand,
$$
\left( \begin{bmatrix} \frac 12 & 0 \\ 0 & 0 \end{bmatrix},
\begin{bmatrix} 0& \frac 34  \\ \frac 34  & 0 \end{bmatrix}\right)
\in\cD_{\De}\setminus\cD_{\Ga},$$
so $\De (X_1,X_2)\succeq0$ does not imply $\Ga (X_1,X_2)\succeq0$.
On the other hand, $\Ga (X_1,X_2)\succeq0$ does imply
$\De (X_1,X_2)\succeq0$. 
We shall prove this later below, see Example \ref{ex:nonScalar2}.
\end{example}

We now introduce  subspaces to be used in our considerations:
\beq\label{eq:spanSdefined}
\cS_j = \span \{I, A_{j,\ell}\mid\ell=1,\ldots,\tg\}
\subseteq
\mbS\R^{d_j\times d_j}.
\eeq

\begin{lem}\label{lem:Sj}
$\cS_j= \span\{\cL_j(X)\mid X\in\R^\tg\} 
.$
\end{lem}

The key tool in studying inclusions of matricial LMI sets is the mapping
$\tau$ we now define.

\begin{definition}\label{def:tau}
Let
$L_1,L_2$ be monic linear pencils as in \eqref{eq:penc12}.
 If
$\{I,A_{1,\ell}\mid \ell=1,\ldots, g\}$
is linearly independent
$($e.g.~$\cD_{\cL_1}$ is bounded$)$,
 we define the unital linear map
\beq\label{eq:tau}
  \tau : \cS_1 \to \cS_2 , \quad A_{1,\ell}\mapsto A_{2,\ell}.
  \eeq
\end{definition}

We shall soon see that,
assuming \eqref{eq:bind1},  $\tau$ has a property called complete positivity, 
which we now introduce.
Let $\cS_j\subseteq\R^{d_j\times d_j}$ be unital linear subspaces
invariant under the transpose, and $\phi:\cS_1\to\cS_2$ a unital linear $*$-map.
For $n\in\N$, $\phi$ induces the map 
$$\phi_n=I_n\otimes\phi:\R^{n\times n}\otimes\cS_1=\cS_1^{n\times n}
 \to \cS_2^{n\times n},
\quad M\otimes A \mapsto M \otimes \phi(A),
$$
called an {\bf ampliation} of $\phi$.
Equivalently,
$$
\phi_n\left( \begin{bmatrix} T_{11} & \cdots & T_{1n} \\
\vdots & \ddots & \vdots \\
T_{n1} & \cdots & T_{nn}\end{bmatrix}\right)
=\begin{bmatrix} \phi(T_{11}) & \cdots & \phi(T_{1n}) \\
\vdots & \ddots & \vdots \\
\phi(T_{n1}) & \cdots & \phi(T_{nn})\end{bmatrix}
$$
for $T_{ij}\in\cS_1$.
We say that $\phi$ is {\bf $k$-positive} if 
$\phi_k$ is a positive map. If $\phi$ is $k$-positive for every $k\in\N$,
then $\phi$ is {\bf completely positive}.
If $\phi_k$ is an isometry for every $k$, then $\phi$ is {\bf 
completely isometric}.

\begin{example}[Example \ref{ex:nonScalar}
revisited]
\label{ex:nonScalar2}
The map $\tau:\cS_2\to\cS_1$ in our example is given by
$$
\begin{bmatrix}1  & 0 \\ 0 & 1
\end{bmatrix}
\mapsto
\begin{bmatrix}  1 & 0 & 0 \\
                     0 & 1 & 0\\
                     0 & 0 & 1
                      \end{bmatrix},\quad
\begin{bmatrix}1  & 0 \\ 0 & -1
\end{bmatrix}
\mapsto
\begin{bmatrix}  0 & 1 & 0 \\
                     1 & 0 & 0\\
                     0 & 0 & 0
                      \end{bmatrix}, \quad
\begin{bmatrix} 0  & 1 \\ 1 & 0
\end{bmatrix}
\mapsto
\begin{bmatrix}  0 & 0 & 1 \\
                     0 & 0 & 0\\
                     1 & 0 & 0
                      \end{bmatrix}.
$$
Consider the extension of $\tau$ to a unital linear $*$-map
$\psi:\R^{2\times2}\to\R^{3\times3}$, defined by
$$
E_{11}\mapsto
\frac 12
\begin{bmatrix}  1 & 1 & 0 \\
                     1 & 1 & 0\\
                     0 & 0 & 1
                      \end{bmatrix}  , \;
E_{12}
\mapsto
\frac 12
\begin{bmatrix}  0 & 0 & 1 \\
                     0 & 0 & 1\\
                     1 & -1 & 0
                      \end{bmatrix}, \;
E_{21} \mapsto
\frac 12
\begin{bmatrix}  0 & 0 & 1 \\
                     0 & 0 & -1\\
                     1 & 1 & 0
                      \end{bmatrix} ,\;
E_{22}
\mapsto
\frac 12
\begin{bmatrix}  1 & -1 & 0 \\
                     -1 & 1 & 0\\
                     0 & 0 & 1
                      \end{bmatrix}.
$$
(Here $E_{ij}$ are the $2\times 2$ matrix units.)
Now we show the map $\psi$ is completely positive.
To do this, we use its Choi matrix defined as
\beq\label{eq:choi2x2}
C=\begin{bmatrix}
\psi(E_{11}) & \psi(E_{12})\\
\psi(E_{21}) & \psi(E_{22})
\end{bmatrix}.
\eeq
\cite[Theorem 3.14]{Pau}
says $\psi$ is completely positive if and only if $C\succeq0$.
We will use the Choi matrix again in Section \ref{sec:comput}
for computational algorithms.
To see that $C$ is positive semidefinite, note
$$
C=\frac 12 W^*W\quad \text{ for }\quad W=
\begin{bmatrix}
1 & 1 & 0 & 0 & 0 & 1 \\
 0 & 0 & 1 & 1 & -1 & 0 \\
\end{bmatrix}.
$$

Now $\psi$ has a very nice
representation:
\beq\label{ex:greatExample}
\psi(S)=\frac12 V_1^*SV_1+\frac 12V_2^*SV_2= \frac 12 \begin{bmatrix} V_1 \\ V_2\end{bmatrix}^*
\begin{bmatrix} S & 0\\ 0 & S\end{bmatrix} 
\begin{bmatrix} V_1 \\ V_2\end{bmatrix}
\eeq
for all $S\in\R^{2\times 2}$.
(Here 
$
V_1=
\begin{bmatrix}
1 & 1 & 0  \\
 0 & 0 & 1 \\
\end{bmatrix}
$
and
$V_2=
\begin{bmatrix}
 0 & 0 & 1 \\
1 & -1 & 0 \\
\end{bmatrix},
$
thus 
$
W = \begin{bmatrix} V_1 & V_2\end{bmatrix}.
$)
In particular, 
\beq\label{ex:great2}
2\De (x,y)= V_1^* \Ga (x,y)V_1+V_2^* \Ga (x,y)V_2.
\eeq
Hence $\Ga (X_1,X_2)\succeq0 $ implies $\De (X_1,X_2)\succeq0$, i.e.,
$\cD_{\Ga }\subseteq\cD_{\De}$.

The formula \eqref{ex:great2} 
illustrates our linear Positivstellensatz which is the subject
of the next subsection.
The construction of the formula in this example 
is a concrete implementation of
the theory leading up to the general result that is presented
in Corollary \ref{cor:tau}. 
\end{example}

\subsection{The map
$\tau$ is completely positive: Linear Positivstellensatz}
\label{sec:tau}
 We begin by equating $n$-positivity of $\tau$ with inclusion
$\cD_{L_1}(n)\subseteq\cD_{L_2}(n)$. Then we use 
the complete positivity of 
$\tau$
to give an algebraic characterization of pencils $L_1$, $L_2$
producing an inclusion $\cD_{L_1}\subseteq\cD_{L_2}$.

\begin{thm}
 \label{thm:tau}
Let
\[
  L_j(x)=I+\sum_{\ell=1}^\tg A_{j,\ell} x_\ell
    \in \mbS\R^{d_j\times d_j}\ax, \quad j=1,2
\]
be monic linear pencils and assume
the matricial LMI set $\cD_{\cL_1}$ is bounded.
Let $\tau:\cS_1\to\cS_2$
be the unital linear map $A_{1,\ell}\mapsto A_{2,\ell}$.
\ben[\rm (1)]
\item
$\tau$ is $n$-positive if and only if $\cD_{L_1}(n)\subseteq
\cD_{L_2}(n)$;
\item
$\tau$ is completely positive if and only if
$\cD_{L_1}\subseteq\cD_{L_2}$;
\item
 $\tau$ is completely isometric if and only 
if $\partial\cD_{L_1}\subseteq\partial\cD_{L_2}$, 
\een
\end{thm}

We remark that the binding condition \eqref{eq:bind2} used in (3)
implies 
\eqref{eq:bind1} used in (2) 
under the boundedness assumption; see
Proposition \ref{prop:eqDom}.  The
proposition says that the relaxed domination problem
(see the abstract) can be restated in terms
of complete positivity, under a boundedness assumption.  
Conversely, suppose $\cD$ is a unital (self-adjoint) subspace
of $\SRdd$ and $\tau:\cD\to \mathbb S\mathbb R^{d^\prime\times d^\prime}$
is completely positive. Given a basis $\{I,A_1,\dots,A_g\}$ 
for $\cD$, let $B_j=\tau(A_j)$. Let
\[
  L_1 = I+\sum A_j x_j, \ \ \ L_2=I+\sum B_j x_j.
\]
 The complete positivity of $\tau$ implies, if
  $L_1(X)\succeq 0$, then $L_2(X)\succeq 0$ and
 hence $\cD_{L_1}\subseteq \cD_{L_2}$. Hence
 the completely positive map $\tau$ (together with a choice
 of basis) gives 
 rise to an LMI domination. 

To prove the theorem we need a lemma.

\begin{lem}\label{lem:bounded34}
Let $L=I+\sum_{j=1}^\tg A_jx_j\in\SRdd\ax$ be a monic linear pencil
with bounded matricial LMI set $\cD_L$.
Then:
\begin{enumerate}[\rm (1)]
\item
  if
 $\Lambda\in \R^{n\times n}$ and $X\in (\SR^{n\times n})^g,$
  and if
\beq\label{eq:T1}
S:=  I\otimes \Lambda + 
L^{(1)}(X)
\eeq
  is symmetric, then  $\Lambda =\Lambda^*$;
\item if $S\succeq0$, then
$\Lambda\succeq 0$;
\item
 if
 $\Lambda\in \R^{n\times n}$ and $X\in (\SR^{n\times n})^g,$
  and if
\beq\label{eq:T2}
T:=  \Lambda \otimes I + \sum_{j=1}^g X_j\otimes A_j 
\succeq0
\eeq
then  $\Lambda\succeq0 $.
\een
\end{lem}

\begin{proof}
 To prove item (1), suppose
\[
  S = I\otimes \Lambda +\sum_{j=1}^g A_j\otimes X_j 
\]
is symmetric. Then
$
0=S-S^* = I\otimes(\Lambda-\Lambda^*).
$
Hence $\Lambda=\Lambda^*$.

For (2),
  if $\Lambda \not \succeq 0$, then there is a vector $v$ such that
$\langle\Lambda v,v \rangle < 0$.
 Consider the projection $P$ onto $\R^d \otimes \R v$, and
let 
$Y=(\langle X_jv,v\rangle)_{j=1}^g\in\R^g$. Then
 the corresponding
compression
$$
PSP= P(I \otimes \Lambda + L^{(1)}(X)  )P
     = I \otimes \langle\Lambda v,v\rangle  +L^{(1)}(Y  ) \succeq 0,
$$
which says that  $L^{(1)}(Y) \succ 0$.
This implies $0\neq t Y\in \cD_\cL$ for all $t>0$;
contrary to $\cD_\cL$ being bounded.

Finally, for (3),
we note that $T$ is, after applying a permutation (often called
the canonical shuffle),
of the form \eqref{eq:T1}. Hence $\Lambda\succeq0$ by (2).
\end{proof}

\begin{proof}[Proof of Theorem {\rm\ref{thm:tau}}]
In each of the three statements, the direction $(\Rightarrow)$ is
obvious. We focus on the converses.

  Fix $n\in\N$. Suppose $T\in \cS_1^{n\times n}$ is positive definite.
Then $T$ is of the form \eqref{eq:T2} for some
$\Lambda\succeq 0$ and $X\in(\SRnn)^g$.
By applying the canonical shuffle,
 \[
   S= I\otimes \Lambda +\sum A_{1,j}\otimes X_j
         \succ 0.
 \]
If we change $\Lambda$ to $\Lambda+\eps I$, the resulting
$T=T_\eps$ is in $\cS_1^{n\times n}$, so without loss of generality
we may assume $\Lambda\succ0$.
   Hence,
 \[
   (I\otimes \Lambda^{-\frac12}) S (I\otimes \Lambda^{-\frac12})
    =  I\otimes I +\sum A_{1,j}\otimes (\Lambda^{-\frac12} X_j
          \Lambda^{-\frac12})  \succ 0.
 \]
  Condition \eqref{eq:bind1} thus says that
 \[
      I\otimes I +\sum A_{2,j}\otimes (\Lambda^{-\frac12} X_j
          \Lambda^{-\frac12})  \succeq 0.
 \]
   Multiplying on the left and right by $I\otimes \Lambda^{\frac12}$
   shows
 \[
    I\otimes \Lambda +\sum A_{2,j}\otimes X_j
         \succeq 0.
 \]
Applying the canonical shuffle again, yields
\[
\tau(T_\eps)=
\Lambda\otimes I +\sum X_{j}\otimes A_{2,j}
         \succeq 0.
\]
  Thus we have proved, if $T_\eps\in \cS_1^{n\times n}$ and $T_\eps\succ 0$,
  then $\tau(T_\eps)\succeq 0$.  An approximation argument now
  shows if $T\succeq 0$, then $\tau(T)\succeq 0$ and
  hence $\tau$ is $n$-positive proving (1). Now (2) follows immediately.

  For (3), suppose $T\in \cS_1^{n\times n}$ has norm one. It follows that
\[
  W=\begin{bmatrix} I & T \\ T^* & I \end{bmatrix} \succeq 0.
\]
  From what has already been proved, $\tau(W)\succeq 0$
  and therefore $\tau(W)$ has norm at most one.  Moreover,
  since $W$ has a kernel, so does $\tau(W)$ and hence
  the norm of $\tau(T)$ is at least one. We conclude
  that $\tau$ is completely isometric.
\end{proof}

\begin{cor}[Linear Positivstellensatz]
 \label{cor:tau}
Let
\[
  L_j(x)=I+\sum_{\ell=1}^\tg A_{j,\ell} x_\ell
    \in \mbS\R^{d_j\times d_j}\ax, \quad j=1,2
\]
be monic linear pencils and assume
$\cD_{\cL_1}$ is bounded.
 If \eqref{eq:bind1} holds, that is, if $\cL_1(X) \succeq 0$
implies $\cL_2(X) \succeq 0$ for all $X$,
 then
there is $\mu\in\NN$ and an isometry $V\in\R^{\mu d_1\times d_2}$ such that
\beq\label{eq:linPos}
L_2(x)= V^* \big( I_\mu \otimes L_1(x) \big) V.
\eeq
Conversely, if $\mu,V$ are as above, then \eqref{eq:linPos}
implies \eqref{eq:bind1} holds.
\end{cor}

\begin{remark}\rm Before turning to the proof of the 
Corollary, we pause for a couple of remarks.
\ben[\rm (1)]
\item
Equation  \eqref{eq:linPos} can be equivalently written
as 
\beq\label{eq:linPss}
L_2(x)=\sum_{j=1}^\mu V_j^* L_1(x) V_j,
\eeq
where
$V_j\in\R^{d_1\times d_2}$ and
$V=
\mbox{col}(V_1,\ldots,V_\mu)$.
Since $\sum_{j=1}^\mu V_j^*V_j=I_{d_2}$, $V$ is an isometry.
Moreover, $\mu$ can be uniformly bounded 
(see the proof of Corollary \ref{cor:tau}, or Choi's characterization
\cite[Proposition 4.7]{Pau}
of completely positive maps between matrix algebras). In fact,
$\mu\leq d_1 d_2$.

\item
Corollary \ref{cor:tau} can be regarded as a Positivstellensatz
for linear (matrix valued) polynomials,
a theme we expand upon later below. 
Indeed, \eqref{eq:linPss} is easily seen to be equivalent to the more
common statement
\beq\label{eq:linPss2}
L_2(x)=B+\sum_{j=1}^\eta W_j^* L_1(x) W_j
\eeq
for some positive semidefinite $B\in\mbS\R^{d_2\times d_2}$
and $W_j\in\R^{d_1\times d_2}$.
\een
\end{remark}

If we worked over $\C$, the proof of Corollary \ref{cor:tau}
would proceed as follows. First invoke 
Arveson's extension theorem \cite[Theorem 7.5]{Pau} to extend
$\tau$ to a completely positive map $\psi$
from $d_1\times d_1$ matrices to $d_2\times d_2$ matrices, and
then 
apply the Stinespring
representation theorem \cite[Theorem 4.1]{Pau} to obtain
\beq\label{eq:taupi}
\psi(a)= V^* \pi(a) V,\quad a\in\CC^{ d_1\times d_1}
\eeq
for some unital
$*$-representation $\pi:\CC^{d_1\times d_1}\to\CC^{d_3\times d_3}$
and isometry (since $\tau$ is unital) $V:\CC^{d_1}\to\CC^{d_3}$.
As all representations of $\CC^{d_1\times d_1}$ are
(equivalent to) a multiple of the identity representation, i.e.,
$\pi(a)= I_\mu\otimes a$ for some $\mu\in\NN$ and all $a\in\CC^{d_1\times d_1}$,
\eqref{eq:taupi} implies \eqref{eq:linPos}.

However, in our case, the pencils $L_j$ have real coefficients 
and we want the isometry $V$
to have real entries as well. 
For this reason and to aid understanding of
this and
our algorithm Section S \ref{sec:comput}
we present a self-contained proof,
keeping all the ingredients real.

We prepare for the proof by reviewing some basic facts about
completely positive maps.
This  serves as a tutorial for LMI experts,
who often are unfamiliar with complete positivity.

Linear functionals  $\sigma: \R^{d_1\times d_1}\otimes  
\R^{d_2\times d_2}\to \mathbb R$ are in a one-one correspondence with 
mappings $\psi:\R^{d_1\times d_1}\to\R^{d_2\times d_2}$ given
by 
 \beq\label{eq:sig2psi}
  \langle \psi(E_{ij})e_a,e_b\rangle 
    = \langle \psi(e_i e_j^*)e_a,e_b\rangle 
    = \sigma(e_j e_i^* \otimes e_a e_b^*).
 \eeq
 Here, with a slight conservation of notation,
 the $e_i,e_j$ are from  $\{e_1,\dots,e_{d_1}\}$ and 
 $e_a,e_b$ are from $\{e_a,\dots,e_{d_2}\}$ which are
  the standard basis for $\mathbb R^{d_1}$ and $\mathbb R^{d_2}$
  respectively.

Now we verify that positive functionals $\sigma$ 
correspond precisely to completely positive $\psi$ and
give a nice representation for such a $\psi$.

\def\tr{\mbox{tr}}
  A positive functional $\sigma:
\R^{d_1\times d_1}\otimes  
\R^{d_2\times d_2} = \R^{d_1d_2\times d_1d_2}
\to \mathbb R$
  corresponds to a positive semidefinite $d_1d_2\times d_1d_2$ matrix $C$ via 
\[
 \sigma(Z)=\tr(ZC).
\]
  Express $C=(C_{pq})_{p,q=1}^{d_1}$ as a $d_1\times d_1$ matrix
  with $d_2\times d_2$ entries.  Thus, the $(a,b)$
  entry of the $(i,j)$ block entry of $C$ is
\[
   (C_{ij})_{ab}=\langle  C (e_j \otimes e_a), e_i\otimes e_b\rangle.
\]
  With 
  $Z=(e_j\otimes e_a)(e_i\otimes e_b)^*$ observe that
\[
  \langle \psi(E_{ij})e_a,e_b\rangle
   =  \sigma(e_je_i^* \otimes e_a e_b^*) 
   =  \tr(ZC) 
   =  \langle C (e_j\otimes e_a),e_i \otimes e_b\rangle 
   =  \langle C_{ij} e_a,e_b\rangle.
\]
 Hence, given $S=(s_{ij})=\sum_{i,j=1}^{d_1} s_{ij} E_{ij}$,
by the linearity of $\psi$,
\[
  \psi(S)=\sum_{i,j} s_{ij}C_{ij}.
\]
(The matrix $C$ is the Choi matrix for $\psi$, illustrated earlier in
\eqref{eq:choi2x2}.)
 The matrix $C$ is positive and thus factors (over the reals) as
 $W^*W$.  Expressing $W=(W_{ij})_{i,j=1}^{d_1}$ as a $d_1\times d_1$ matrix
 with $d_2\times d_2$ entries $W_{ij}$, 
\[
  C_{ij} =\sum_{\ell=1}^{d_1} W_{j\ell}^* W_{i\ell}.
\]
 Define $V_\ell = (W_{i\ell})$. Then we have
$\sigma$ positive implies
\beq\label{eq:reppsi}
 \psi(S)=  \sum_{i,j=1}^{d_2} s_{ij} C_{ij}
     =  \sum_{\ell=1}^{d_1}  \sum_{i,j=1}^{d_2} W_{i\ell}^* s_{ij}  W_{j\ell} 
     =  \sum_{\ell=1}^{d_1} V_\ell^* S V_{\ell} 
     = V^* \big ( (I_{d_1}\otimes S)\otimes I_{d_2} \big ) V,
\eeq
where $V$ denotes the column with $\ell$-th entry $V_\ell$.
Hence $\psi$ is completely positive.

\begin{proof}[Proof of Corollary {\rm\ref{cor:tau}}]
We now proceed to prove Corollary \ref{cor:tau}.
Given $\tau$ as in Theorem \ref{thm:tau},
  define 
a linear functional
$\tilde\sigma: \cS_1\otimes\R^{d_2\times d_2}\to \mathbb R$ 
as in correspondence \eqref{eq:sig2psi}
by
 \[
  \tilde\sigma(S\otimes Y) =\sum_{a,b} \langle Ye_b,e_a\rangle 
          \langle \tau(S)e_b,e_a\rangle.
 \]
  Suppose $Z=\sum S_k\otimes Y_k \in \cS_1 \otimes \R^{d_2\times d_2}$ 
is positive semidefinite
  and let $\mathbf e =\sum_{a=1}^{d_2} e_a\otimes e_a$.
 Since
the map $\tau_{d_1}=I_{d_1}\otimes \tau$, called
an ampliation of $\tau$,  
is positive,
\[
 0\le  \langle \tau_{d_1}(Z)\mathbf e, \mathbf e\rangle 
    =  \tilde\sigma(Z).
\]
  Thus $\tilde\sigma$ is positive and hence extends to a positive
  mapping $\sigma: \R^{d_1\times d_1}\otimes  \R^{d_2\times d_2}\to \mathbb R$ 
 by the Krein extension theorem, which in turn corresponds
to a completely positive mapping 
$ \psi:\R^{d_1\times d_1}\to\R^{d_2\times d_2}
$ 
as in \eqref{eq:sig2psi}.
It is easy to verify that $\psi|_{\cS_1}=\tau$.
By the above, 
\[
 \psi(S)= V^* \big ( ( I_{d_1}\otimes S)\otimes I_{d_2}\big ) V.
\]
Since $\psi(I)=I$, it follows that
 $V^*V=I$.
\end{proof}

\subsection{Equal matricial LMI sets}\label{sec:preequalD}

In this section we begin an analysis of the binding condition \eqref{eq:bind2}.
We present an equivalent reformulation:

\begin{prop}\label{prop:eqDom}
Let $L_1$, $L_2$ be monic linear pencils.
 If $\cD_{L_1}$ is bounded and
\eqref{eq:bind2} holds, that is, if
$\partial\cD_{L_1}\subseteq\partial\cD_{L_2}$,
then
$
\cD_{\cL_1}=\cD_{\cL_2}.
$
\end{prop}

The proof is an easy consequence of the following elementary
observation on convex sets.

\begin{lem}\label{lem:convexSets}
Let $C_1\subseteq C_2\subseteq \RR^n$ be closed convex sets, $0\in\Int C_1\cap\Int C_2$.
If $\partial C_1\subseteq \partial C_2$ then
$C_1=C_2$.
\end{lem}

\begin{proof}
By way of contradiction, assume $C_1\subsetneq C_2$ and let
$a\in C_2\setminus C_1$. The interval $[0,a]$ intersects $C_1$
in $[0,\mu a]$ for some $0<\mu<1$. Then $\mu a\in\partial C_1
\subseteq \partial C_2$. Since $0\in\Int C_1$, $C_1$ contains
a small disk $D(0,\ep)$. Then $K:= \co (D(0,\ep) \cup\{a\})$
is contained in $C_2$ and $\mu a\in \Int K\subseteq\Int C_2$
contradicting $\mu a\in\partial C_2$.
\end{proof}

\begin{proof}[Proof of Proposition {\rm \ref{prop:eqDom}}]
Let $C_i:=\cD_{\cL_i}$, $i=1,2$. Then
$$
\partial C_i=\{X\in\cD_{\cL_i}\mid \cL_i(X)\succeq0,  \cL_i(X)\not\succ0\}.
$$
Since $C_1$ is closed and bounded, it is the convex hull of its boundary.
Thus by \eqref{eq:bind2}, $C_1\subseteq C_2$.
Hence the assumptions of Lemma \ref{lem:convexSets}
are fulfilled and we conclude $C_1=C_2$.
\end{proof}

\begin{example}
It is tempting to guess that $\cD_{\cL_1}=\cD_{\cL_2}$ implies $\cL_1$
and $\cL_2$ (or, equivalently, $L_1^{(1)}$ and $L_2^{(1)}$) are 
unitarily equivalent.
In fact, in the next subsection we will show this to be true
under a certain irreducibility-type assumption.
However, in general this fails for the trivial reason that the
direct sum of a representing pencil and an ``unrestrictive'' pencil
is also representative.

Let $L_1$ be an arbitrary monic linear pencil (with
$\cD_{\cL_1}$ bounded) and
$$L_2(x)=I+ \big(L_1^{(1)}(x)\oplus \frac 12 L_1^{(1)}(x)\big)=
\begin{bmatrix} I+L_1^{(1)}(x) &0\\ 0& I+\frac 12 L_1^{(1)}(x) \end{bmatrix}=
\begin{bmatrix} L_1(x) &0\\ 0& I+\frac 12 L_1^{(1)}(x) \end{bmatrix}
.$$
Then $\cD_{\cL_1}=\cD_{\cL_2}$ but $L_1$ and $L_2$ are obviously
not unitarily equivalent.
However,
$$
L_1(x)= \begin{bmatrix} I \\ 0 \end{bmatrix}^* L_2(x) \begin{bmatrix} I \\ 0 \end{bmatrix}
$$
in accordance with Corollary \ref{cor:tau}.

Another guess would be that under $\cD_{\cL_1}=\cD_{\cL_2}$, we
have $p=1$ in Corollary \ref{cor:tau}. However this
example also refutes that. Namely, there is no isometry
$V\in\R^{d_1 \times 2d_1}$
satisfying
$$
\begin{bmatrix} L_1(x) &0\\0& I+ \frac 12 L_1^{(1)}(x) \end{bmatrix} = L_2(x)= V^* L_1(x) V.
$$
(Here $L_1$ is assumed to be a $d_1\times d_1$ pencil.)
\end{example}

\subsection{Minimal $L$ representing $\cD_L$ are unique: Linear
Gleichstellensatz}
\label{subsec:minimal}\label{subsec:equalD}

Let $L=I+\sum A_ix_i$ be a $d\times d$ monic linear pencil and
$\cS= \span \{I, A_{\ell}\mid\ell=1,\ldots,\tg\}$.
In this subsection we explain how to
associate a monic linear pencil $\tL$ to $L$ with the following
properties:
\ben[\rm (a)]
\item
$\cD_\tcL=\cD_\cL$;
\item
$\tL$ is the minimal (with respect to the size of the defining
matrices) pencil satisfying (a).
\een

A pencil $\tL=I+\sum \tA_j x_j$ is a {\bf subpencil} of $L$
provided there is a nontrivial reducing subspace $\cH$ for $\cS$ such
that $\tA_j = V^* A_j V$, where $V$ is the inclusion of
$\cH$ into $\R^d$, where $d$ is the size of the matrices $A_j$.
The pencil $L$ is {\bf minimal} if there does not exist
a subpencil $\tL$ such that $\cD_\cL=\cD_{\tcL}$.

\begin{thm}
 \label{thm:minimal}
Suppose $L$ and $M$ are linear pencils of size $d\times d$
 and $e\times e$ respectively. If $\cD_\cL=\cD_M$
 is bounded and  both $L$ and $M$ are minimal,
 then $d=e$ 
  and there is a unitary $d\times d$ matrix
 $U$ such that $U^* L U=M$; i.e., $L$ and $M$
 are unitarily equivalent.

 In particular, all minimal pencils for a given matricial LMI set
 have the same size $($with respect to the defining
 matrices$)$ and this size is the smallest possible.
\end{thm}

\begin{example}
Suppose $L$ and $M$ are only minimal with respect to the
spectrahedra $\cD_L(1)$ and $\cD_M(1)$, respectively.
Then $\cD_L(1)=\cD_M(1)$ does not imply that $L$ and $M$
are unitarily equivalent.
For instance, 
let $L$ and $M$ be the two pencils studied in Example \ref{ex:nonScalar}.
Then both $L$ and $M$ are minimal, $\cD_L(1)=\cD_M(1)$,
but $L$ and $M$ are clearly not unitarily equivalent.
\end{example}

 The remainder of this subsection is devoted to the
 proof of, and corollaries to, Theorem \ref{thm:minimal}.
We shall see how $\cD_\cL$ is governed by the multiplicative
structure (i.e., the $C^*$-algebra) $C^*(\cS)$ generated by $\cS$ as
well as the embedding $\cS\hookrightarrow C^*(\cS)$.
For this we borrow heavily from Arveson's noncommutative
Choquet theory \cite{Arv1,Arv2,Arv3} and to a lesser
extent from  the
paper of the third author with Dritschel \cite{DM}.

We start with a basics of 
real $C^*$-algebras needed in the proof of Theorem \ref{thm:minimal}.
First, the well-known classification result.

\begin{prop}\label{prop:fdimcstar}
A finite dimensional real $C^*$-algebra is $*$-isomorphic
to a direct sum of real $*$-algebras of the form
$M_n(\R)$, $M_n(\C)$ and $M_n(\HH)$.
$($Here the quaternions $\HH$ are endowed with the standard involution.$)$
\end{prop}

\begin{prop}\label{prop:realIso}
Let $\K\in\{\R,\C,\HH\}$ and let $\Phi:M_n(\K)\to M_n(\K)$ be
a real $*$-isomorphism.
\ben[\rm (1)]
\item
If $\K\in\{\R,\HH\}$, then there exists a unitary $U\in M_n(\K)$ 
with $\Phi(A)=U^*AU$ for all $A\in M_n(\K)$.
\item
For $K=\C$, there exists a unitary $U\in M_n(\C)$ 
with $\Phi(A)=U^*AU$ for all $A\in M_n(\C)$ or
$\Phi(A)=U^* \bar A U$ for all $A\in M_n(\C)$. 
$($Here $\bar A$ denotes the entrywise complex conjugate of $A$.$)$
\een
\end{prop}

\begin{proof}
In (1), $M_n(\K)$ is a central simple $\R$-algebra. By the Skolem-Noether
theorem \cite[Theorem 1.4]{KMRT}, there exists an invertible matrix $U\in M_n(\K)$ with
\beq\label{eq:realIso1}
\Phi(A)=U^{-1}AU \quad \text{ for all }\quad A\in M_n(\K).
\eeq
Since $\Phi$ is a $*$-isomorphism,
$$
U^{-1} A^* U = \Phi(A^*) =
\Phi(A)^*= \left(U^{-1}AU \right)^* = U^* A^* U^{-*},
$$
leading to $UU^*$ being central in $M_n(\K)$. By scaling,
we may assume $UU^*=I$, i.e., $U$ is unitary.

(2) $\Phi(i)$ is central and a skew-symmetric matrix, hence
$\Phi(i)=\alpha i$ for some $\alpha\in\R$. Moreover,
$\Phi(i^2)=-1$ yields $\alpha^2=1$. So $\Phi(i)=i$ or 
$\Phi(i)=-i$. In the former case, $\Phi$ is a $*$-isomorphism
over $\C$ and thus given by a unitary conjugation as in (1).
If $\Phi(i)=-i$, then $\Phi$ composed with entrywise conjugation is 
a $*$-isomorphism
over $\C$. Hence there is some unitary $U$ with 
$\Phi(A)=U^* \bar A U$ for all $A\in M_n(\C)$.
\end{proof}

\begin{remark}\label{rem:realComplIso}
For $K\in\{\R,\C,\HH\}$, 
every real $*$-isomorphism
$\Phi:M_n(\K)\to M_n(\K)$ 
lifts to a 
unitary conjugation isomorphism 
$M_{dn}(\R)\to M_{dn}(\R)$, where 
$d=\dim_\R\K$.
By Proposition \ref{prop:realIso}, this is
clear if $K\in\{\R,\HH\}$. To see why this is true in the complex case
we proceed as follows.

Consider the standard real presentation of complex matrices,
induced by 
\beq\label{eq:Ciota}
\iota:\C\to M_2(\R), \quad 
a+i\, b \mapsto \begin{bmatrix} a & b \\ -b & a\end{bmatrix}.
\eeq
If the real $*$-isomorphism
$\Phi:M_n(\C)\to M_n(\C)$ 
is itself a unitary conjugation, the claim is obvious.
Otherwise 
$\bar\Phi$ 
is
 conjugation by some unitary $U\in M_n(\C)$ and thus has 
a natural extension to a
$*$-isomorphism
$$\check\Phi: M_{2n}(\R)\to M_{2n}(\R), \quad A \mapsto \iota(U)^* A \iota(U).$$
Then
$$
\hat\Phi: M_{2n}(\R)\to M_{2n}(\R), \quad A\mapsto
\left(I_n\otimes\begin{bmatrix}1& 0 \\ 0 & -1\end{bmatrix}\right)^{-1} \check\Phi(A) \left(I_n\otimes\begin{bmatrix}1& 0 \\ 0 & -1\end{bmatrix}\right)
$$
is a unitary conjugation $*$-isomorphism of $M_{2n}(\R)$ and 
restricts to $\Phi$ on $M_n(\C)$.
\end{remark}

Let $K$ be the biggest two sided ideal of $C^*(\cS)$ such
that the natural map
\beq\label{eq:isoProj}
C^*(\cS) \to C^*(\cS)/K, \quad a\mapsto\ta:=a+K
\eeq
is completely isometric on $\cS$. 
 $K$ is called the {\bf \v Silov ideal} (also the
{boundary ideal}) for $\cS$ in $C^*(\cS)$. Its existence 
 and uniqueness is nontrivial,
see the references given above.
The snippet \cite{Arv4} contains a streamlined, compared
to approaches which use injectivity, presentation
of the \v Silov ideal based upon the existence of
completely positive maps with the unique extension property.
While this snippet, as well as all of the references in
the literature of which we are aware, use complex scalars,
the proofs go through with no essential changes in the real
case. 

 A {\bf central projection} $P$ in $C^*(\cS)$ is a projection
 $P\in C^*(\cS)$ such that $PA=AP$ for all $A\in C^*(\cS)$
 (alternately $PA=AP$ for all $A\in \cS$).
 We will say that a projection $Q$ {\bf reduces} or is 
  {\bf a reducing projection for} $C^*(\cS)$ if $QA=AQ$ for all $A\in C^*(\cS)$.
 In particular, $P$ is a central projection if $P$ reduces
 $C^*(\cS)$ and $P\in C^*(\cS)$. 

\begin{prop}\label{prop:minimal}
 Let $L$ be a $d\times d$ truly linear pencil and
 suppose $\cD_\cL$ is bounded. 
 Then $L$ is minimal if and only if
\begin{enumerate}[\rm (1)]
 \item  every minimal reducing projection $Q$ is in fact 
   in $C^*(\cS)$; and 
 \item the \v Silov ideal of $C^*(\cS)$ is $(0)$.
\end{enumerate}
\end{prop}

\begin{proof}
Assume (1) does not hold
and let $Q$ be a given minimal nonzero reducing projection 
  for $C^*(\cS)$, which is not an element of $C^*(\cS)$.
  Let $P$ be a given minimal nonzero central projection
  such that $P$ dominates $Q$; i.e., $Q\preceq P$.
By our assumption, $Q\neq P$.

  Consider the real
  $C^*$-algebra $\cA= C^*(\cS)P$ as a real $*$-algebra
  of operators on the range $\mathcal H$ of $P$.  First we claim
  that the mapping $\cA \ni A \mapsto AQ$ is one-one.
  If not, it has a nontrivial kernel $J$ which is
  an ideal in $\cA$. 
  The subspace $\mathcal K=J\mathcal H$
  reduces $\cA$
   and moreover, because of finite dimensionality, 
   the projection $R$ onto $\mathcal K$ is in fact in
   $\cA$.
  Hence, $R$ is a central
  projection. 
  By minimality, $R=P$ or $R=(0)$.
  In the second case the mapping is one-one.
  In the first case,  $J\cH=\cH$ and
  thus $J=C^*(\cS)P$; i.e., the
  mapping $C^*(\cS)P \ni A \mapsto AQ$ is identically
  zero. In this case, the mapping 
  $C^*(\cS)P \ni A \mapsto A(I-Q)$ is completely isometric,
  contradicting the minimality of $L$. 
  Hence the map $\cA\ni A\mapsto AQ$ is indeed one-one.

  Therefore, the mapping $C^*(\cS)\ni A\mapsto A(I-P)+AQ$
  is faithful and in particular completely isometric. 
  Thus the restriction of our pencil to the span of 
  the ranges of $I-P$ and $Q$ produces a
  pencil $L'$ with $\cD_{L'}=\cD_L$,
  but of lesser dimension. Thus, we have proved,
  if (1) does not hold, then $L$ is not minimal. 

  It is clear that if the \v Silov ideal of $C^*(\cS)$ is nonzero,
  then $L$ is not minimal. Suppose $J\subset C^*(\cS)$
  is an ideal and the quotient
  mapping $\sigma:\cS\to C^*(\cS)/J$ is 
  completely isometric.  As before, let
  $\cK=J\mathbb R^d$ (where the pencil $L$ has size $d$).
  The projection $P$ onto $\cK$ is a central projection.
  Because for $S\in \cS$ we have both $\sigma(S)=\sigma(S-SP)$,
  and $\sigma$ is completely isometric, it follows that
  $S\mapsto S(I-P)$ is completely isometric. By the 
  minimality of $L$, it follows that $P=0$. 

Conversely, suppose (1) and (2) hold. If $L$ is not minimal, let $\tL$ 
denote a minimal subpencil
with $\cD_{\tL} = \cD_L,$ corresponding to a reducing subspace
$\cK \subsetneq \R^d$ for $\cS.$ 
 Let $Q$ denote the projection onto $\cK$ and $\cT$ denote $\{SQ\mid S\in
\cS\}.$ Note that the equality $\cD_{\tL}=\cD_L$ says exactly that the
 mapping $\cS \to \cT$ given by
$S\mapsto SQ$ is completely isometric.
In particular, if $R$ is the projection onto a reducing subspace which contains
$\cK$, then also $S\mapsto SR$ is completely isometric.

Let $P:\R^d\to\cK'$ denote any minimal orthogonal projection 
onto a reducing subspace of $\cK^\perp$.
By (1),
$P\in C^*(\cS)$, and hence $C^*(\cS)P$ is a (minimal)
two-sided ideal of $C^*(\cS)$. 
On the other hand, $(I-P)$ is the projection onto a reducing
subspace which contains $\cK$ and hence $S\mapsto S(I-P)$
is completely isometric.  Now let $S=(S_{i,j}) \in M_n(\cS)$ be given.
If $T=(T_{i,j})(I_n\otimes P) \in M_n(C^*(\cS))P$, then
\[
 \| S+T\| = \|S(I_n\otimes (I-P)) \oplus (S+T)(I_n\otimes P) \|
 \ge \| S(I_n\otimes (I-P))\| = \| S\|,
\]
where the last equality comes from the fact that $S\mapsto S(I-P)$
is completely isometric and the inequality from the fact that
the norm of a direct sum is the maximum of the norm of the summands.
Of course choosing $T=S(I_n\otimes P)$ it follows that the norm of $S$ in the quotient
 $C^*(\cS)/C^*(\cS)P$ is the same as $\| S\|$.
Hence the  induced map
$\cS\to C^*(\cS)/ C^*(\cS)P$
is completely isometric
and therefore  $C^*(\cS)P$
is contained in the \v Silov ideal of $\cS$, contradicting
(2).
\end{proof}

\begin{proof}[Proof of Theorem {\rm\ref{thm:minimal}}]
  Write $L=I+\sum A_j x_j$ and $M=I+\sum B_j x_j$ and
  let $C^*(\cS)$ and $C^*(\cT)$ denote the unital
 $C^*$-algebras generated by $\{A_1,\dots,A_g\}$
 and $\{B_1,\dots,B_g\}$ respectively.
 By Proposition \ref{prop:minimal}, both $C^*(\cS)$ and $C^*(\cT)$
 are reduced relative to $\cS$ and $\cT$ respectively; i.e.,
 the \v Silov ideals for $\cS$ and $\cT$ respectively are $(0)$.

  Moreover, for $\cQ$ and $\cP$ maximal families of 
  minimal nonzero reducing projections for $C^*(\cS)$ and $C^*(\cT)$
 respectively,
we use Proposition \ref{prop:minimal} to obtain
\[
   C^*(\cS)=\oplus_{Q\in\cQ} C^*(\cS)Q,\quad
   C^*(\cT)=\oplus_{P\in\cP} C^*(\cT)P.
\]
For later use we note that a minimal ideal in these $C^*$-algebras
is of the form $C^*(\cS)Q$ for $Q\in\cQ$, and
$C^*(\cT)P$ for $P\in\cP$, respectively.

 The unital linear $*$-map
$$
\tau: \cS\to \cT, \quad A_j\mapsto B_j
$$
is a completely isometric isomorphism by Theorem \ref{thm:tau} and maps between
reduced operator systems. By \cite[Theorem 2.2.5]{Arv1}, $\tau$
is induced by a $*$-isomorphism
$$
  \rho: C^*(\cS) \to C^*(\cT).
$$
  Since $\rho$ is an isomorphism
 of $C^*$-algebras and $C^*(\cS)P$ for $P\in\cP$, is a minimal ideal, 
\beq\label{eq:isostarmin}
\rho(C^*(\cS)P)=C^*(\cT)Q
\eeq 
for some $Q\in\cQ$.
The converse is true too.
That is, for each $Q\in\cQ$ there is a unique
$P\in\cP$ such that \eqref{eq:isostarmin} holds.
 We conclude that $d=e$.

By Proposition \ref{prop:realIso} and Remark \ref{rem:realComplIso} 
we also conclude that the $C^*$-isomorphism $\rho:C^*(\cS)P\to C^*(\cT)Q$
  must be implemented by a unitary mapping $\Ran P\to \Ran Q$.
\end{proof}

\begin{cor}\label{cor:choq2}
Let $L\in\mbS\R^{d\times d}\ax$ be a monic linear pencil with bounded
$\cD_L$ and
$\tL\in\mbS\R^{\ell\times\ell}\ax$ its minimal pencil. Then there is
a $(d-\ell)\times(d-\ell)$ monic linear pencil 
$J$ satisfying $J|_{\cD_\cL}\succeq0$
and a unitary $U\in\R^{d\times d}$ such that
$$
L(x)= U^* \begin{bmatrix}
\tL(x) \\ & J(x) \end{bmatrix}
U.
$$
\end{cor}

\begin{proof}
Easy consequence of the construction of $\tL$.
\end{proof}

\section{Computational algorithms}
\label{sec:comput}
 In this section we present several numerical algorithms using
 semidefinite programming (SDP) \cite{WSV}, based on
 the theory developed in the preceding section.
 However, one can read and implement these algorithms without
 reading anything beyond Section \ref{sec:algo} of the introduction.
 In each case, we first present the algorithm
 and then give the justification (which a user need not read).
 The following section, Section \ref{sec:more-cube}, 
 provides comparisons and refinements
 of the matricial matrix cube algorithm of Subsection \ref{subsec:cube}
 below. 

 Given $L_1$ and $L_2$ monic linear pencils
\beq
\label{eq:pencils12}
 L_j(x) = I+\sum_{\ell=1}^\tg A_{j,\ell} x_\ell \in\mbS\R^{d_j\times d_j}\ax , \quad j=1,2,
\eeq
with bounded matricial LMI set $\cD_{L_1}$,
we present an algorithm, {\bf the inclusion algorithm},
to test whether $\cD_{L_1}\subseteq\cD_{L_2}$.
Of course this numerical
test yields a sufficient condition
for containment of the spectrahedra $\cD_{L_1}(1) \subseteq \cD_{L_2}(1)$.
We refer the reader to Section \ref{subsec:bd} for a test of boundedness of
LMI sets, which works both for commutative LMIs and matricial LMIs, and
computes the radius of a matricial LMI set based 
 on the basic inclusion algorithm. 
 Subsection \ref{sec:reduce}
 contains a refinement of the  basic inclusion algorithm,
 in the case that either $L_1$ or $L_2$ is a direct sum
  of pencils of smaller size.  As an application,  
   we then present a matricial version of the classical matrix
 cube problem 
 in Section \ref{subsec:cube}. 
  Analysis of the matricial matrix cube
  algorithm are in Section   \ref{sec:more-cube} along with
 a comparison to the matrix cube  algorithm  
 of Ben-Tal and Nemirovski \cite{B-TN}. 
  There 
  further algorithms, which offer improved estimates, at
  the expense of additional computation, for the matrix cube
  problem are also discussed. 
The final subsection of this section
 gives a (generically successful) algorithm
for computation of a minimal representing pencil and
the \v Silov ideal, 
these being the only algorithms
whose statement is not self contained.

\subsection{Checking inclusion of matricial LMI sets}\label{subsec:incl}

\def\jus{ \noindent {\bf Justification.}}

\def\ss{\smallskip}
\def\ms{\medskip}
\def\bs{\bigskip}

\begin{center}
{\bf The inclusion algorithm}
\end{center}
Given: $A_{1,\ell}$ and  $A_{2,\ell}$ $ \text{ for }\quad \ell=1,\ldots,g$.
Let $\alpha_{p,q}^\ell$ denote the $(p,q)$ entry of
$A_{1,\ell}$.

\ss
\noindent
Solve
the following (feasibility) SDP:
\beq
\label{eq:tausdp}
(c_{pq})_{p,q=1}^{d_1} := 
C\succeq0, \qquad 
\sum_p^{d_1} c_{pp} = I_{d_2},
\qquad
\forall \ell=1,\ldots,g: \; 
\sum_{p,q}^{d_1}  \alpha_{pq}^\ell
c_{pq}=  A_{2,\ell},
\eeq
 for the unknown symmetric matrix $C$. 
 Since each 
  $c_{pq}$ is a $\R^{d_2 \times d_2}$
 matrix,
the  symmetric  matrix $C$ of unknown variables (reasonably termed the Choi matrix) 
is of size $d_1d_2\times d_1d_2$ and there are
$\frac 12 d_1d_2(d_1d_2+1)$ (scalar) unknowns and
$\frac 12 (1+g)d_2(d_2+1)$ (scalar) linear equality constraints.
This can be, in practice, solved numerically with standard SDP solvers. 
 In the next subsection, we show that if $L_1$ has special structure,
   then the number of ($C$) variables can be reduced, sometimes
  dramatically.  

\ss
\noindent
Conclude:
$\cD_{L_1}\subseteq\cD_{L_2}$ if and only if the SDP
\eqref{eq:tausdp} is feasible, i.e., has a solution.

\jus \ \
By Theorem \ref{thm:tau} and Corollary \ref{cor:tau}, $L_2$ is 
positive semidefinite on $\cD_{L_1}$ 
if and only if
there is a completely positive unital map 
\beq\label{eq:tauu}
\tau:\R^{d_1\times d_1}\to\R^{d_2\times d_2}
\eeq
 satisfying
\beq\label{eq:tauuu}
\tau(A_{1,\ell})=A_{2,\ell} \quad \text{ for }\quad \ell=1,\ldots,g.
\eeq

To determine the existence of such a map,
consider the Choi matrix $C=\big(\tau(E_{ij})\big)_{i,j=1}^{d_1}\in
(\R^{d_2\times d_2})^{d_1\times d_1}$ of 
$\tau$. (Here, $E_{ij}$ are the $d_1\times d_1$ elementary matrices.)
For convenience of notation we consider $C$ to be a $d_1\times d_1$ matrix
with $d_2\times d_2$ entries $c_{ij}$.
This is the matrix $C$ which appears in the algorithm.
It is well-known that $\tau$ is completely positive if and only if $C$ is positive semidefinite
\cite[Theorem 3.14]{Pau}.

Note that we can write
$A_{1,\ell}= \sum_{p,q} \alpha_{pq}^\ell E_{pq}$.
Then
$\tau(A_{1,\ell})=\sum_{p,q}\alpha_{pq}^\ell \tau(E_{pq}) = 
\sum_{p,q}\alpha_{pq}^\ell c_{pq}$.
This lays behind the last equation in \eqref{eq:tausdp}.
 If a solution $C$ to \eqref{eq:tausdp} has been obtained, 
 then a Positivstellensatz-type certificate for
the inclusion of the matricial LMI sets
$\cD_{L_1}\subseteq\cD_{L_2}$
can be obtained; cf.~Example \ref{ex:nonScalar2}
or the proof of Corollary \ref{cor:tau}.
\qed

\subsection{LMIs which are direct sums of LMIs}
 \label{sec:reduce}
   If either  pencil $L_j$ as in \eqref{eq:pencils12} 
   is given as a direct sum of pencils, 
   then the Choi matrix  $C$
  in the inclusion algorithm can be chosen 
  with many fewer unknowns,  reflecting this structure. 
  We start with $L_1$.
  
\def\del{\delta}
\begin{prop}
 \label{prop:reduce}
   Suppose $L_1 =\oplus_{\mu=1}^k M_\mu$, where $M_1,\dots, M_k$ are monic linear
   pencils,
\[
   M_\mu = I + \sum_{\ell=1}^g B_\ell^\mu x_\ell,
\]
  where the $B_j^\mu$ are of size $\del_\mu \times \del_\mu$.
  Thus, $A_{1,\ell} = \oplus_{\mu=1}^k  B_\ell^\mu$. 
  Let $\alpha_{pq}^{\ell,\mu}$ denote the $(p,q)$ entry 
  of $B_\ell^\mu$. 
  Then,
  $\cD_{L_1}\subseteq \cD_{L_2}$ if and only if there exists a
  symmetric matrix $C=\oplus_{\mu=1}^k C ^\mu$ such that 
\beq
\label{eq:tausdp-red}
\begin{split}
\forall \mu=1,\ldots,k: \qquad
C^\mu:= (c_{pq}^\mu)_{p,q=1}^{\del_\mu} & \succeq0,  \\
 \forall \ell =1,\ldots,g: \qquad 
\sum_{\mu=1}^k \sum_{p,q=1}^{\delta_\mu} 
\alpha_{pq}^{\ell,\mu} c_{pq}^\mu & =  A_{2,\ell}, \\
\sum_{\mu=1}^k   \sum_{p=1}^{\delta_\mu}  c_{pp}^\mu & = I_{d_2}.
\end{split}
\eeq
  Each $c^\mu_{pq}$ is an unknown  $d_2 \times d_2$ matrix
   and $(c^\mu_{pq})^*=(c^{\mu}_{qp})$. 
\end{prop}

\begin{proof}
 The inclusion $\cD_{L_1}\subseteq \cD_{L_2}$ is equivalent to the existence of
 a Choi matrix $C$ satisfying the feasibility conditions 
\eqref{eq:tausdp}
of the inclusion 
 algorithm.  Thus $C$ is a $d_1\times d_1$ block matrix with $d_2\times d_2$
 entries. On the other hand, $d_1=\sum_\mu \del_\mu$ and the matrix 
 $C$ can be viewed
 as a block matrix $C=(C_{i,j})_{i,j}^k$ where $C_{i,j}$ is 
 $\del_i \times\del_j$ block matrix whose
  entries are  $d_2 \times d_2 $ matrices.
   Observe that for $i\ne j$, the entries of $C_{i,j}$
  do not appear as part of the linear constraint in the inclusion
  algorithm - they are unconstrained because our direct sum structure
  forces certain $\alpha_{pq}$ to be zero.  

  Since $d_1=\sum \delta_\mu$ and the matrix $C$ is a
  $d_1\times d_1$ block matrix with $d_2 \times d_2$ blocks 
  and is 
  positive semidefinite,  hence  there exist
  $d_1\times \delta_\mu$ block matrices $W_\mu$ 
  having $d_2 \times d_2$ entries  such that $C$ factors as 
\[
   C =W^* W = \begin{bmatrix} W_1^* \\ \vdots \\ W_k^*\end{bmatrix} 
        \begin{bmatrix} W_1 &\cdots & W_k \end{bmatrix}
\]
  Consider the set $\mathcal C$ 
  of $2^{k-1}$ matrices of the form
\[
   \begin{bmatrix} W_1^* \\ \pm W_2^* \\ \pm W_3^* \\ \vdots \\ \pm W_k^* \end{bmatrix}
     \begin{bmatrix} W_1 &\pm W_2 & \pm W_3 & \cdots & \pm W_k \end{bmatrix}.
\]
  Each $\tilde{C}\in \mathcal C$  
   solves the inclusion algorithm; i.e., validates 
  $\cD_{L_1}\subseteq \cD_{L_2}$.  
  Hence the matrix $\hat{C}$ obtained by averaging
  over $\mathcal C$ also validates the inclusion. Noting that,
  because each off diagonal entry of $\hat{C}$ is the average of
  $2^{k-2}$ terms $W_i^* W_j$ with $2^{k-2}$ terms $-W_i^* W_j$, 
  we get
  $\hat{C}$ is the block diagonal matrix with diagonal entries $W_j^* W_j$,
  which completes the proof. 
\end{proof}

  With the hypotheses of Proposition \ref{prop:reduce}, 
  the number of unknown variables in the LMI inclusion
  algorithm are greatly reduced.  
  Indeed, from 
  $\frac{1}{2} (d_1d_2 +1 )d_1d_2$, to  
  $$\frac{1}{2} \sum_{\mu=1}^k( d_2\delta_\mu+1) d_2 \delta_\mu.$$ 
   The number of equality constraints
  is still  $\frac12 (1+g)d_2(d_2 +1 ).$

  A reduction in both the number of variables and equality constraints
  occurs if $L_2$, the range linear pencil, in the inclusion algorithm
  has a direct sum structure.

\begin{prop}
 \label{range-sum}
    In the inclusion algorithm, if the pencil 
   $L_2$ is a direct sum; i.e., $L_2 =\oplus_{\mu=1}^k M_\mu$,
   where each 
\[
  M_\mu = I+\sum_1^g B^\mu_j x_j
\]
    is a monic linear pencil
   of size $\delta_\mu \times \delta_\mu$ $($so that
   $\sum \delta_\mu = d_2)$, then
   $\cD_{L_1}\subset \cD_{L_2}$ if and only if there exists
   a symmetric matrix $C=\oplus_{\mu=1}^k C ^\mu$ such that 
\beq
\label{eq:tausdp-red-range}
\begin{split}
\forall \mu=1,\ldots,k: \qquad
C^\mu:= (c_{pq}^\mu)_{p,q=1}^{d_1} & \succeq0,  \\
 \forall \ell =1,\ldots,g,\ \ \mu=1,\ldots,k: \qquad 
\sum_{p,q=1}^{d_\mu} 
\alpha_{pq}^{\ell} c_{pq}^\mu & =  B^\mu_\ell, \\
\forall \mu=1,\ldots,k: \qquad
\sum_{p=1}^{d_\mu}  c_{pp}^\mu & = I_{\delta_\mu}.
\end{split}
\eeq
  Each $c^\mu_{pq}$ is an unknown  $\delta_\mu \times \delta_\mu$ matrix
   and $(c^\mu_{pq})^*=(c^{\mu}_{qp})$. 
\end{prop}
The count of unknowns is
  $\frac{1}{2} \sum_{\mu=1}^k( d_1 \delta_\mu+1) d_1 \delta_\mu$ 
and of scalar equality constraints is 
$ g k \delta_\mu( \delta_\mu+1)+k \delta_\mu( \delta_\mu+1).$

\subsection{Tightening the relaxation}
\label{sec:tight}
  There is a general approach to tightening
  the  inclusion  algorithm which relaxes
  $\cD_{L_1}(1) \subseteq \cD_{L_2}(1)$, and
  thus applies to the algorithms in the section, 
  based upon the following simple lemma.

\begin{lem}
 \label{lem:simple}
  Suppose $L_1,L_2$ and $M$ are 
linear pencils and let 
  $\hat{M}=L_1 \oplus M.$  If  
  $\cD_{L_1}(1)\subseteq \cD_{M}(1),$ then
 \[
    \cD_{\hat{M}} \subseteq \cD_{L_1}
\quad\text{ and }\quad
   \cD_{\hat{M}}(1) = \cD_{L_1}(1). 
 \]
  In particular, if $\cD_{\hat{M}}\subseteq \cD_{L_2}$, 
   then $\cD_{L_1}(1) \subseteq \cD_{L_2}(1)$. 
\end{lem}

\def\hM{{\hat M}}

\begin{proof}
  The first part of the lemma is evident:
$$\cD_{\hat M}=\cD_{L_1}\cap \cD_{M}\subseteq\cD_{L_1}.$$
Likewise, $\cD_{\hat M}(1)=\cD_{L_1}(1)\cap \cD_{M}(1) 
= \cD_{L_1}(1)$, since $\cD_{L_1}(1)\subseteq \cD_{M}(1)$.
For the last statement note that 
$\cD_{\hat{M}}\subseteq \cD_{L_2}$
implies
$ \cD_{L_1}(1) = \cD_{\hat M}(1)\subseteq \cD_{L_2}(1)$.
\end{proof}

  This lemma tells us applying our inclusion algorithm to $\hM$
  versus $L_2$ is at least as accurate as applying it to $L_1$ versus
  $\hM$ and it quite possibly is more accurate.
  The lemma is used in the context of the matrix cube problem
  in Section \ref{sec:more-cube}. 

\subsection{Computing the radius of matricial LMI sets}\label{subsec:bd}

Let 
$L$ be a monic linear pencil,
\beq
\label{eq:pencils1}
 L(x) = I+\sum_{\ell=1}^\tg A_{\ell} x_\ell \in\mbS\R^{d\times d}\ax.
\eeq
We present an algorithm based on semidefinite programming 
to compute the radius of a matricial LMI set $\cD_{L}$ (and at the same
time check
whether it is bounded).
The idea is simply to use the test in Section \ref{subsec:incl}
to check if $\cD_L$ is contained in the ball of radius $N$.
The smallest such $N$ will be the matricial radius, and 
also an upper bound on the radius of the
spectrahedron $\cD_L(1)$.

\def\CR{\color{red}}

Let 
$$
\cJ_N(x)=
\frac 1N
\begin{bmatrix}
N & x^* \\
x & N I_g
\end{bmatrix}
= I+\frac 1N \sum_{j=1}^{g}  (E_{1,j+1}'+E_{j+1,1}') x_j \in\mbS\R^{(g+1)\times(g+1)}\ax
$$
be a monic linear pencil. 
 Here $E'_{ij}$ the $(g+1)\times (g+1)$ elementary matrix
  with a $1$ in the $(i,j)$ entry and zeros elsewhere. 

Then $\cD_L$ is bounded, and its matricial radius is $\leq N$, if and only if
$\cD_{\cJ_N}\supseteq\cD_L$.

\begin{center}
{\bf The matricial radius algorithm}
\end{center}
Let $\alpha_{r,s}^\ell$ denote the $(r,s)$ entry of
$A_{\ell}$, that is,
$A_{\ell}= \sum_{r,s} \alpha_{rs}^\ell E_{rs}$.
Solve the SDP (RM):
\begin{enumerate}[\rm (RM$_1$)\;]
\item[]\label{eq:ba0}
$\min 
b:= \sum_{r,s}
\alpha_{rs}^1 (c_{rs})_{1,2}\quad$ subject to\\
\item\label{eq:ba1}
$(c_{rs})_{r,s=1}^d:= C\succeq0$, 
\item\label{eq:ba2}
 $\sum\limits_{r=1}^d c_{rr}=I_{g+1},$
\item\label{eq:ba3}
$\forall \ell=1,\ldots, g, \; 
\forall p,q=1,\ldots,g+1:$ 
$$\sum_{r,s}
\alpha_{rs}^\ell (c_{rs})_{p,q}=0 \quad \text{ for }\quad (p,q)\not\in
\{ (1,\ell+1), \,(\ell+1,1)\},$$ 
\item\label{eq:ba4}
$
\sum_{r,s}
\alpha_{rs}^1 (c_{rs})_{1,2}=
\sum_{r,s}
\alpha_{rs}^1 (c_{rs})_{2,1}=
\sum_{r,s}
\alpha_{rs}^2 (c_{rs})_{1,3}=
\sum_{r,s}
\alpha_{rs}^2 (c_{rs})_{3,1}=
\cdots$ \\
${}\quad =
\sum_{r,s}
\alpha_{rs}^g (c_{rs})_{1,g+1}=
\sum_{r,s}
\alpha_{rs}^g (c_{rs})_{g+1,1}$
\een
for the unknown $C$; i.e., the $d^2$  unknown 
 $(g+1)\times (g+1)$ matrices $(c_{rs})$.
If the  optimal value of (RM)
is $b\in\R_{>0}$, then $\|X\|\leq \frac 1b$
for all $X\in\cD_L$, and this bound is sharp.

This SDP is always feasible (for $b=\sum_{r,s}
\alpha_{rs}^1 (c_{rs})_{1,2}=0$).
Clearly, $\cD_{L}$
is bounded
if and only if this SDP
has a positive solution.
In fact, any value of $b>0$ obtained gives
an upper bound of $\frac 1b$ for the norm of an element in $\cD_L$.
The size of the (symmetric) matrix of unknown variables 
is $d(g+1)\times d(g+1)$ and there are
$\frac 12 (g^3+4g^2+3g+4)$ (scalar) linear constraints.
To reduce the number of unknowns, solve the linear system
of $\frac 12g(g^2+3g-2)$ equations given in (RM$_3$).

Checking boundedness of $\cD_L(1)$ is a classical,
fairly basic semidefinite programming problem.
Indeed, given a nondegenerate monic linear pencil $L$, $\cD_L(1)$
is bounded (equivalently, $\cD_L$ is bounded)
if and only if the following SDP is infeasible:
$$
L^{(1)}(X)\succeq0,\quad \tr \left(L^{(1)}(X)\right)=1.
$$
(Here, $L^{(1)}$ denotes the truly linear part of $L$.)

However, computing the radius of $\cD_L(1)$ is harder.
Thus our algorithm, yielding a convenient upper bound on the radius,
 might be of broad interest, 
motivating us to spend more time describing its implementation.
The algorithm can be written entirely in a 
matricial form which is both elegant and easy to code
in MATLAB or Mathematica.
The matricial component of the algorithm is as follows.
Let ${\mathbf e}_n$ denote the vector of length $n$ with
all ones, 
let ${\mathbf E}_n= {\mathbf e}_n\otimes {\mathbf e}_n^t$ be the
$n\times n$ matrix of all ones.
Then 
(RM$_2$)
is (using $\hsp$ for the
Hadamard product) 
$$
\big( {\mathbf e}_{g+1} \otimes I_d \big)^t \big(
( I_d \otimes {\mathbf E}_{g+1})
\hsp C
 \big)  \big( {\mathbf e}_{g+1} \otimes I_d \big) = I_{g+1},
$$
while the left hand side of (RM$_3$) 
can be presented as
the $(p,q)$ entry of
$$
\big( {\mathbf e}_{g+1} \otimes I_d \big)^t \big(
( A_\ell \otimes {\mathbf E}_{g+1})
\hsp C
 \big)  \big( {\mathbf e}_{g+1} \otimes I_d \big).
$$
Equations 
(RM$_3$) 
and 
(RM$_4$)
give constraints on these matrices.

As an example we computed the matricial radius of an ellipse,
which for the example we computed agrees with the scalar radius.
 The corresponding
Mathematica notebook can be downloaded from\\
\centerline{\url{http://srag.fmf.uni-lj.si/preprints/ncLMI-supplement.zip}}

\ss

\jus \ \
As in the previous subsection, we need to determine whether
there is a completely positive unital map $\tau:\Rdd\to\R^{(g+1)\times(g+1)}$
satisfying $\tau(A_j)=\frac 1N (E_{1,j+1}'+E_{j+1,1}')$
for some $N$.
The Choi matrix here is $C=(\tau(E_{ij}))_{i,j}\in
(\R^{(g+1)\times (g+1)})^{d\times d}$.
Let $A_\ell=\sum_{r,s} \alpha_{rs}^\ell E_{rs}$.
Then the linear constraints we need to consider say that
$$
\tau(A_\ell)=\sum_{r,s} \alpha_{rs}^\ell c_{rs}
$$
has all entries $0$ except for the $(1,\ell+1)$ and $(\ell+1,1)$ entries which
are the same; indeed they are all equal to $\frac 1N$. Thus
we arrive at 
the feasibility SDP (RM) above.
\qed

\subsection{The matricial matrix cube problem}
\label{subsec:cube}
 This section describes our matricial matrix cube algorithm - a test 
  for  inclusion of the matricial matrix cube (as defined below) 
  into a given LMI set.  Variations on the
  algorithm and an analysis of the connection
  between this algorithm and the matrix cube algorithm of \cite{B-TN}
  is the subject of Section \ref{sec:more-cube}. 

Let 
$L\in\mbS\R^{d\times d}\ax$ be a monic linear pencil as in \eqref{eq:pencils1}.
We present an algorithm that computes the size $\rho$ of 
 the biggest matricial cube contained in $\cD_L$. 
That is, $\rho\in\R$ is the largest number with the
following property:
 if $n\in\N$ and 
$X\in (\SRnn)^g$ satisfies $\|X_i\|\leq\rho$ for all $i=1,\ldots,
g$, then $X\in\cD_L$.
When $X_i$ is in $\R^{1\times 1}$ this is 
the classical
matrix cube problem (cf.~Ben-Tal and Nemirovski \cite{B-TN}),
which they show is NP-hard.

First we need an LMI which defines the cube.
Let
$$
\cC_\rho(x)=\frac 1\rho \left( 
   \big(\oplus_{j=1}^g \rho - x_j\big)\bigoplus \big(\oplus_{j=1}^g \rho+x_j\big)
\right)\in\mbS\R^{2g\times 2g}\ax.
$$
Then $\cC_\rho(x)= I + \frac 1\rho
 \sum_{j=1}^g ( E_{jj}-E_{g+j,g+j}) x_j$, 
 where  $E_{i,j}$ is a the elementary $2g\times 2g$  
  matrix with a $1$ in the $(i,j)$ entry and zeros elsewhere,
  and
$$\cD_{\cC_\rho}
= \bigcup_{n\in\N} \left\{X\in(\SRnn)^g \mid \;  \|X_i\|\leq\rho\text{ for all }i=1,\ldots,g\right\}.
$$
This is a matricial cube.
Our algorithm uses the test in Section \ref{subsec:incl}
to compute the largest $\rho$ with $\cD_{\cC_\rho}\subseteq \cD_L$.
 It also takes advantage of the fact that $\cC_\rho$ 
  is a direct sum (of scalar-valued pencils)
    by using Proposition \ref{prop:reduce} with 
 $ k=2g$,
    $\delta_\mu=1$ and $d_2=d$.   This 
  immediately gives rise to the following SDP:

\begin{enumerate}[\rm (preMC$_1$)\;]
\item[]
\label{eq:cu1}
$\max \rho\quad$ subject to\\
\item
 $C^j\succeq0$,  \quad $j=1, \ldots, 2g$
\item\label{eq:precu2}
$\forall j=1,\ldots, g:\quad$
$
C^j-C^{g+j} = \rho A_j.
$
\item\label{eq:precu3}
 $\sum\limits_{j=1}^{2g} C^j =I_{d},$

\een
Each of the $2g$ symmetric matrices
$C^j$ is in ${\mathbb S} \R^{d \times d}$.

Next we make this algorithm more efficient by
solving the equality constraints 
(preMC$_2$)
to eliminate $C^{g+1}, \ldots, C^{2g}$ and 
(preMC$_3$)
to obtain
\beq
\label{eq:mc2eq}
C^g = \frac12 \Big(I - 2 \sum_{j=1}^{g-1} C^j + \rho \sum_{j=1}^g A_j  \Big)  
\eeq

With this, the above SDP reduces to

\begin{center}
{\bf The matricial matrix cube algorithm}
\end{center}

\begin{enumerate}[\rm (MC$_1$)\;]
\item[]
$\max \rho\quad$ subject to\\
\item
\label{eq:cu1a}
 $C^j\succeq0$,  \quad $j=1, \ldots, g-1$
\item
\label{eq:cu2a}
$C^j \succeq \rho A_j$  
\item
\label{eq:cu3a}
$I_d -  2 \sum_j^{g-1} C^j + \sum_j^{g-1}\rho A_j
 \pm \rho A_g   \succeq 0,$
\end{enumerate}
where each of the $g-1$ symmetric matrices
$C^j$ is in ${\mathbb S} \R^{d \times d}$.

This SDP is always feasible (with $\rho=0$). 
If its optimal value
is $\rho>0$, then $\cD_{\cC_\rho}\subseteq \cD_L$, and the obtained
upper 
bound for the size of the matricial cube is sharp.
There are
   $\frac12 (g-1) d(d+1)$ variables and
  all of the linear equality constraints have been eliminated.  
There are $2g$ matrix inequality constraints.

\begin{example}
\label{ex:nonScalar3}
Consider finding the largest  square embedded inside
the unit disk. 
 We consider the two pencils
$\De, \Ga$ from 
Example \ref{ex:nonScalar}, each of which represents the unit disk,
 since $\cD_{\De}(1)=\cD_{\Ga}(1)=\{(X_1,X_2)\in\R^2\mid X_1^2+X_2^2\leq1\}$.
It is clear that $\cD_{\cC_{\sqrt2/2}}(1)$
is the maximal square contained in the unit disk $\cD_{\De}(1)$. 
Indeed the biggest matricial cube in $\cD_{\De}$ is
$\cD_{\cC_{\sqrt2/2}}$, but 
the biggest matricial cube in $\cD_{\Ga}$ is
$\cD_{\cC_{  \frac 1 2 }}$.
For details, see the Mathematica notebook 
available at\\
\centerline{\url{http://srag.fmf.uni-lj.si/preprints/ncLMI-supplement.zip}}
We will revisit this example in Section \ref{sec:more-cube}.
\end{example}


\jus \ \
 A justification for the matrix cube algorithm based on the pre-algorithm
 has already been given.  So it suffices to justify the pre-matricial
  matrix cube algorithm.  
Let $B_j:= E_{j,j} - E_{g+j,g+j}\in\mbS\R^{2g\times2g}$.
 Taking advantage of the fact that $\cC_\rho$ 
 is the direct sum of $2g$ scalar linear pencils, we 
want to determine the biggest $\rho$ for which there exists
a completely positive unital map 
 $\tau: \oplus_1^{2g} \R^{1\times 1}\to\R^{d\times d}$
satisfying $\tau(B_j)=\rho A_j$, $j=1,\ldots,g$.
Suppose 
  $C=\oplus_{j=1}^{2g} (c^j) \in \oplus_1^{2g} (\R^{d\times d})^{1\times 1}$ 
 (because each $c^j$ is a $1\times 1$ block matrix whose
  entries are (symmetric) $d\times d$ matrices, there is
  no need for the indexing $c_{pq}^j$)
 is the corresponding Choi matrix as in Proposition \ref{prop:reduce}.
Then the linear constraint $\tau(B_j)=\rho A_j$ translates into
$
c_{1,1}^j - c_{1,1}^{g+j} = \rho A_j
$
which is 
(preMC$_2$).
\qed

\subsection{Minimal pencils
 and the \v Silov ideal}
\label{subsec:silov}

This section describes an algorithm aimed at constructing
from a given pencil $L$ a pencil $\tL$ of minimal size with
$\cD_L=\cD_\tL$.

\subsubsection{Minimal pencils}
Let 
$L$ be a monic linear pencil,
\beq
\label{eq:pencils13}
 L(x) = I+\sum_{\ell=1}^\tg A_{\ell} x_\ell \in\mbS\R^{d\times d}\ax
\eeq
with bounded $\cD_L$.
We present a probabilistic algorithm 
based on semidefinite programming that computes a minimal
pencil $\tilde L$ with the same matricial LMI set. 

The two-step procedure goes as follows. 
In Step 1, one uses 
the decomposition of a semisimple
algebra into a direct sum of simple algebras, a classical technique
in computational algebra, cf.~Friedl and Ro\' nyal \cite{FR}, 
Eberly and Giesbrecht \cite{EG}, or 
 Murota, Kanno, Kojima, and Kojima 
\cite{MKKK} for a recent treatment.
This yields a unitary matrix $U\in\Rdd$ that
simultaneously
 transforms
the $A_\ell$ into block diagonal form,
that is,
$$U^*A_\ell U= \oplus_{j=1}^s B_\ell^j\quad\text{ for all }\quad \ell.$$
For each $j$, the set $\{I,B_1^j,\ldots, B_g^j\}$ generates a simple
real algebra. Define the monic linear pencils
$$
L^j(x)=I+\sum_{\ell=1}^g B_\ell^j x_\ell, \quad
L'(x)=U^* L(x) U= \oplus_{j=1}^s L^j(x).$$ 
Given $\ell$, let $\tL_\ell=\oplus_{j\ne \ell} L^j$.
If there is no $\ell$ such that
\[
   L^\ell|_{\cD_{\tL_\ell}}\succeq 0,
\]
 (this can be tested using SDP as explained in Section \ref{subsec:incl})
 then the pencil is minimal. If there is such an $\ell$
 remove  the (one) corresponding block from $L'$
  to obtain a new pencil
  and repeat the process.
Once we have no more redundant blocks in $L'$, the obtained pencil
$\tilde L$ is minimal, and satisfies $\cD_{\tilde L}=\cD_L$ by
construction.

\subsubsection{\v Silov ideal}

 Thus subsection requires material from Section \ref{sec:arv}.
Using our results from
Section \ref{subsec:minimal} (cf.~Proposition \ref{prop:minimal})
 and Section \ref{subsec:incl},
one can compute the \v Silov ideal of a unital matrix algebra $\cA$
generated by symmetric matrices $A_1,\ldots,A_g\in\mbS\Rdd$.
Form the monic linear pencil
$$
L=I+\sum A_\ell x_\ell \in\mbS\Rdd\ax,
$$
and
compute the 
minimal pencil 
$$\tilde L=I+\sum \tilde A_\ell x_\ell$$
as in the previous
subsection. 
If $$\tilde S=\span \{I,\tilde A_{\ell}\mid\ell=1,\ldots,g\},$$
then the kernel of the canonical unital map
$$
\cA\to C^*(\tilde S), \quad A_\ell\mapsto\tilde A_\ell
$$
is the \v Silov ideal
of $\cA$.

\section{More on the matrix cube problem}
 \label{sec:more-cube}
  This section provides perspective on the inclusion algorithm
 by focusing on the matrix cube problem.
  The first subsection
 shows that the estimate based on the inclusion
 algorithm, namely the matricial
 matrix cube algorithm of Subsection \ref{subsec:cube},
 is essentially identical to that obtained by the algorithm
 of Ben-Tal and Nemirovski
 in \cite{B-TN}.

 Subsection \ref{sec:not-practical},
 illustrates the tightening procedure
 of Lemma \ref{lem:simple} on the matricial cube.

\subsection{Comparison with the algorithm in \cite{B-TN}}
 \label{sec:compare}
  Let $L=I+\sum_{\ell=1}^\tg A_{\ell} x_\ell\in\mbS\R^{d\times d}\ax$ 
 be a monic 
 linear pencil and recall the pencil $\cC_\rho$ and
  its corresponding positivity domain
 $\cD_{\cC_\rho}$, the matricial cube.
 In \cite{B-TN} the {\it verifiable sufficient condition} for the inclusion
 $\cD_\rho(1)\subseteq\cD_L(1)$ is the following: 
Suppose 
there exist symmetric
  matrices $B_1,\dots B_g$ such that
\beq\tag{S}\label{eq:nemBT}
B_j \succeq \pm A_j \text{ for all } j=1,2,\dots,g; \quad\text{and}\quad
 I-\rho \sum_j B_j \succeq 0
\eeq
holds. Then 
  $\cD_{\cC_\rho}(1)\subseteq \cD_L(1)$.

  The following proposition says that the 
  estimate of the largest cube contained
  in a given spectrahedron given by
  the matricial relaxation based upon
  the  matricial matrix cube algorithm  
  is the same as that based upon 
  condition \eqref{eq:nemBT}.

\begin{prop}
 \label{prop:we-are-nem}
    Given $\rho\in\R_{\geq0}$, 
   condition \eqref{eq:nemBT} holds 
   if and only if  $\cD_{\cC_\rho} \subseteq \cD_L$.
Moreover, there is an explicit formula for converting 
  condition \eqref{eq:nemBT}  to a feasible point for
the matricial matrix cube algorithm
  and vice-versa. 
\end{prop}

\begin{proof}
Suppose we have found the optimal $\rho$ and 
the corresponding
$C^j$
for $j=1, \ldots, 2g$,
in the matricial matrix cube algorithm.
From 
(preMC$_2$)  $\rho A_j = C^j- C^{g+j}$. 
 Set 
\[
 \rho B^j:=  C^j + C^{g+j}.
\]
 From (preMC$_1$),
\beq
\label{eq:BC}
\rho (B^j-A_j) = 2C^{j}\succeq 0  \quad  \text{ and }
\quad
 \rho (B^j+A_j) = 2 C^{j+g} \succeq 0.
\eeq
 Relation (preMC$_3$) gives
$I = \rho  \sum_j B^j$.
Thus we see $B^j $ and $\rho$ satisfy  \eqref{eq:nemBT}.

 Conversely, suppose $B_j, \rho$ are a solution 
 to \eqref{eq:nemBT}. 
 Solve \eqref{eq:BC} for $C^j,C^{g+j}$.  
 It is straightforward to check that these $C^j$
  satisfy the conditions (preMC$_j$) for $j=1,2,3$. 
\end{proof}

Now that we know the estimate provided by our relaxation is the same as that
of algorithm \eqref{eq:nemBT}, we look  at the computational cost.
 \eqref{eq:nemBT} has $\frac 12g d(d+1)$ unknowns and the number of 
$d\times d$
 matrix
 inequality constraints is  $(2g+1)$. 
 As we saw, our matricial matrix cube algorithm
  had $\frac 12 (g-1) d(d+1)$ unknowns and $2g$
 matrix ($d \times d $) inequality constraints,
 so the costs are a bit less than those of  \eqref{eq:nemBT}.
However,  \eqref{eq:nemBT} can be improved easily
by the general trick in the following remark
which removes an unknown and a constraint, 
thus making the cost of  \eqref{eq:nemBT} 
 the same as ours.

\begin{remark}
\rm
 \label{rem:trace0}
    If, in the $A_{1,\ell}$ in the inclusion
  algorithm of Section \ref{subsec:incl} all have trace $0$, then the condition
  $I_{d_2}-\sum_{p=1}^{d_1} c_{pp}=0$ is equivalent to the inequality
   $\Delta = I_{d_2}- \sum_{p=1}^{d_1} c_{pp}\succeq 0$, since 
   then the $c_{pp}$ can, without harm, be replaced by 
 $c_{pp}+\frac{\Delta}{d_1}.$  When presented with the inequality form
 we could convert it to equality, 
   then eliminate one variable by solving for it.
\end{remark}

\subsection{The lattice of 
inclusion algorithm relaxations for the matrix cube}
\label{sec:not-practical}
  A virtue of the general method of 
  the inclusion algorithm based  on matricial  relaxations is that,
  as alluded to in Section \ref{sec:tight}, it
  allows tightening in order to improve 
  estimates  (with added cost). 
  This subsection discusses and illustrates
  properties of  this tightening procedure,
 mainly as an introduction to a topic that
 might merit  further study.
 We do show in an example that
 tightening can produce an improved estimate.

\subsubsection{General theory}
  The operator theory upon which this paper is based,
  when converted to the language of  LMIs,
  contains a theory of matricial relaxations of a given LMI
  set and thus provides a general framework containing 
   the tightening
  methods described in Section \ref{sec:tight}.

  Suppose $S\subseteq \mathbb R^g$
  is an LMI set; i.e., suppose
  there is a monic linear pencil $\Lambda$ 
  such that  $S=\cD_\Lambda(1)$.   The collection 
  $\mathcal L_S$  of all monic 
  linear pencils $L$ with $S=\cD_L(1)$ is naturally ordered
  by inclusion. Namely,  if $L,M\in\mathcal L_S$, then $L\ge M$
  if,  $\cD_L(n)\subset \cD_M(n)$ for every positive integer $n$.
  If $\Lambda^\prime$ is also a monic linear pencil and
   $S\subset S^\prime =\cD_{\Lambda^\prime}$, then 
  the matricial inclusion $\cD_M \subset \cD_{\Lambda^\prime}$
  implies the inclusion $\cD_L \subset \cD_{\Lambda^\prime}$.
  Thus, the pencil $L$ gives 
  at least as good a test for the inclusion $S\subset S^\prime$,
  than does $M$.  Similarly, if $M^\prime, L^\prime\in\mathcal L_{S^\prime}$
  and $L^\prime \ge M^\prime$, then $M^\prime$ gives at least as good a 
  test for the inclusion of $S$ into $S^\prime$ as does $L^\prime$.

  If $\mathcal L_S$ has a maximal element, denote it by $L_{\max}$
  and similarly for $L_{\min}$.  Generally, 
  $\mathcal L_S$ will not have either a minimal or maximal element; 
  however, it turns out for the matrix cube there
  is a minimal element. See Proposition 
  \ref{prop:nem-is-worst} below.  Further, in general by  dropping 
  the requirement that the pencils $L$ have matrix coefficients
  and instead allowing for  operator coefficients, it is possible
  to prove that $L_{\max}$ and $L_{\min}$ exists.  (The discussion
  in \cite{Pau} on max and min operator space structures
  is easily seen to carry over to the present setting.)


  Thus, though typically not practical, 
  the matricial relaxation  for the inclusion 
  of the set $S=\cD_{\Lambda}(1)$ into $T=\cD_{\Lambda^\prime}$  
  based upon using $L_{\max}$ in place of $\Lambda$
  produces the exact answer; 
  whereas the matricial relaxation based upon 
  $\cD_{L_{\min}}$ produces the most conservative estimate
  over all possible matricial relaxations.

  \subsubsection{$L_{\min}$ and $L_{\max}$  for the matrix cube}
\label{ss:min-max}
From the next proposition it follows that
for the matrix cube $\cD_{\cC_1}(1)$ the minimal pencil
 $L_{\min}$ is $\cC_1$.

\begin{prop}
 \label{prop:nem-is-worst}
   If $M$ is a monic linear pencil and $\cD_M(1)\subseteq \cD_{\cC_\rho}(1)$,
   then $\cD_M \subseteq \cD_{\cC_\rho}$. In particular, 
   if $\cD_M(1)=\cD_{\cC_\rho}(1)$, and $L$ is a monic linear pencil 
   for which $\cD_{\cC_\rho} \subseteq \cD_L$, then $\cD_M \subseteq \cD_L$.
   Hence if $\cD_M(1)=\cD_{\cC_1}(1)$, then the inclusion
   $\cD_M\subset \cD_L$ is at  least as good a test for 
   $\cD_L$ to contain the unit cube as the inclusion 
   $\cD_{\cC_1}\subset \cD_L$. 
\end{prop}

\begin{proof}
  Write $M=I- \sum A_j x_j$. The condition 
  $\cD_M(1)\subseteq \cD_{\cC_\rho}(1)$ implies, if
\[
  \frac{1}{\rho} \sum A_j x_j  \preceq I,
\quad\text{for some } x_j\in\R,
\]
  then $|x_j|\le 1$ for each $1\le j \le g$. 

  Now suppose that $X=(X_1,\dots,X_g)\in (\SRnn)^g$ and
\[
  \frac{1}{\rho} \sum A_j \otimes X_j \preceq \rho I.
\]
  For each vector $f$ and unit vector $x$ (of the appropriate sizes),
  it follows that
\[
  \frac{1}{\rho} \sum \langle A_j f,f\rangle 
           \langle X_j x, x\rangle \le \|f\|^2. 
\]
  With $x$ fixed, varying $f$ shows that
\[
   \frac{1}{\rho} \sum A_j \langle X_j x,x \rangle \preceq I.
\]
  It now follows that $| \langle X_j x,x \rangle |\le 1$ 
  for each $j$ and unit vector $x$. Hence, $\|X_j\|\le 1$
  for each $j$ and hence $X\in\cD_{\cC_\rho}(n)$
  and the proof is complete.  
\end{proof}

  We have not computed $L_{\max}$
  for the matrix cube ($g>2$ variables), but we have for the matrix square 
  ($g=2$ variables)
  and found it to be a pencil with operator (infinite dimensional)
  coefficients. We do not give the calculation in this paper,
  rather we content ourselves with the simple,
  and natural,  example below 
  which suffices to show that  there are in fact choices
  of $M$ in Proposition \ref{prop:nem-is-worst} which do
  lead to improved estimates for the matrix cube
  problem and  that $L_{\min}$ and $L_{\max}$ 
  are different for the cube.
  Of course, any such improved estimate comes with additional
  computational cost; and, because it is in only two variables
 (where solving four LMIs gives the exact answer), 
  the example is purely illustrative.

 Given $\eta=(s,t)\in\R^2$ with $s^2+t^2 =1$, let
$$
  A_1(\eta) = \begin{bmatrix} s & 0 \\ 0 & -s\end{bmatrix}, \quad
  A_2(\eta) =\begin{bmatrix} 0 & t\\ t & 0 \end{bmatrix},
  \qquad \quad
  L_\eta = I  + \sum A_j(\eta) x_j.
 $$
  Recall, a unitary matrix which is symmetric 
 is called a {\bf signature matrix}.  
  Up to scaling the $A_j(\eta)$ are 
  signature matrices, and 
  further  $(A_1(\eta)\pm A_2(\eta))^2 =I$.
 
  It is straightforward to show that  $\cD_{L_\eta}(1)$
  contains the unit square; i.e.,
  $\cD_{\cC_1}(1)\subset \cD_{L_\eta}(1)$.
  Hence, with   $M_\eta =\cC_1 \oplus L_\eta,$ 
\[ 
  \cD_{M_\eta}(1)= \cD_{\cC_1}(1)
\]
  and at the same time,
\begin{equation}
 \label{ex-eta}
  \cD_{M_\eta} \subset \cD_{\cC_1}. 
\end{equation}
  On the other hand, for $\eta\ne (\pm 1,0)$
  or $(0,\pm 1)$, it is possible to check by hand 
  that $\cD_{\cC_1}(n)\not\subseteq \cD_{L_\eta}(n)$ 
  for each $n\ge 2$.    Indeed, 
  let $X=(X_1,X_2)$  denote the tuple of $2\times 2$ matrices  
\[
 X_1 = \begin{bmatrix} 1 & 0 \\
 0 & -1
\end{bmatrix}, \quad \text{ and }\quad
X_2= \begin{bmatrix}
 0 & 1 \\
 1 & 0
\end{bmatrix}.
\]
  Then $(X_1,X_2)\in \cD_{\cC_1}(2)$, but
  $(X_1,X_2) \notin\cD_{L_\eta}(2)$. 
  Hence, the inclusion 
  in equation \eqref{ex-eta} is proper. 

  Another way to see the inclusion is proper is to verify that 
  the extreme points $X=(X_1,X_2)$  of $\cD_{\cC_1}(n)$ are exactly the
  pairs of $n \times n $
  signature matrices $X_1,X_2$.
  On the other hand, 
    for $\eta\not \in \{(\pm 1,0),(0,\pm 1)\}$,
  the extreme points $X=(X_1,X_2)$ of $\cD_{\cC_1}(n)$ which are also  in 
   $\cD_{L_\eta}(n)$  are precisely the  pairs of $n \times n $
  signature matrices $X_1,X_2$ which commute.

\begin{example}[Example \ref{ex:nonScalar3} revisited]\label{ex:nonScalar4}
Recall the pencil $\Gamma$ from Example \ref{ex:nonScalar}. 
In
Example \ref{ex:nonScalar3} we  employed the matricial matrix
cube relaxation $\cD_{\cC_\rho}\subset\cD_{\Gamma}$ 
to obtain a lower bound of $\frac 12$ for 
the biggest square inside $\cD_{\Gamma}(1)$. 
To tighten the relaxation we direct sum $L_\eta$ to $\cC_\rho$ to 
obtain a linear pencil $M_\eta$. Hand calculations
for this problem tell us that
with $\eta=( \sqrt 2/2, \sqrt 2/2)$ we obtain the exact 
relaxation $\cD_{M_\eta}\subseteq \cD_{\Gamma}$.
However, suppose we did not know this and ask:
 will selecting  $\eta$ without much care
  give a reasonable improvement?

We made  100 runs with random $\eta$ and
considered the inclusion
$\rho \cD_{M_\eta}\subset\cD_{\Gamma}$
for each $\eta$
and found the average value for 
$\rho$ 
to be approximately $0.6$.
This is a considerable improvement over $0.5$
obtained in the untightened problem.
\end{example}

\section{Positivstellens\"atze on a matricial LMI set}
\label{sec:pos}

We give an algebraic characterization of symmetric
polynomials $p$ in noncommuting variables
with the property that $p(X)$ is positive definite
for all $X$ in a bounded matricial LMI set $\cD_L$.
The conclusion of 
this  Positivstellensatz is stronger than previous ones
because of the stronger hypothesis that  $\cD_L$ is an LMI set.
If the polynomial $p$ is linear, then an algebraic characterization is
given by Theorem \ref{thm:intro1}.
We shall  use the linear Positivstellensatz, Corollary \ref{cor:tau},
to prove
that the quadratic module associated to a
monic linear pencil $L$ with bounded
$\cD_L$ is archimedean.
Thereby we obtain a Putinar-type Positivstellensatz \cite{Put} without the
unpleasant added ``bounding term''.
In this section, for simplicity of presentation we stick to
polynomials on the free $*$-algebra. Later in Section \ref{sec:morePosSS}
we give this improved type of Positivstellensatz on  general $*$-algebras,
a few examples being commuting variables, free variables
and free symmetric variables (this section).
The material here is motivated by the
study of positivity of matrix polynomials in \emph{commuting}
variables undertaken in \cite{KlSw}; see also \cite{HM}.

To state and prove our next string of results, we need to introduce
notation pertaining to words and polynomials in 
noncommuting variables.

\subsection{Words and NC polynomials}\label{subsec:NCpoly}
Given positive integers $n,d,d^\prime$ and $g$,
  let $\Mdd$ denote the $d^\prime \times d$ matrices
  with real entries and $(\Mnns)^{g}$
  the set of $g$-tuples of real $n\times n$ matrices.

We write $\ax$ for the monoid freely
generated by $\x=(x_1,\ldots, x_g)$, i.e., $\ax$ consists of {\bf words} in the $g$
noncommuting
letters $x_{1},\ldots,x_{g}$
(including the empty word $\emptyset$ which plays the role of the identity $1$).
Let $\R\ax$ denote the associative
$\R$-algebra freely generated by $\x$, i.e., the elements of $\R\ax$
are polynomials in the noncommuting variables $\x$ with coefficients
in $\R$. Its elements are called {\bf (NC) polynomials}.
An element of the form $aw$ where $0\neq a\in \R$ and
$w\in\ax$ is called a {\bf monomial} and $a$ its
{\bf coefficient}. Hence words are monomials whose coefficient is
$1$.
Endow $\R\ax$ with the natural {\bf involution}
fixing $\R\cup\{\x\}$ pointwise.
The involution reverses words. For example, 
$(3-  2 x_{1}^2 x_{2} x_{3})^\T =3  -2 x_{3} x_{2} x_{1}^2.$

\subsubsection{NC matrix polynomials}
More generally, 
for an abelian group $R$ we use $R\langle x\rangle$
  to denote the abelian
  group of all (finite)
  sums of {\bf monomials} in $\ax$.
Besides $R=\R$, the most important example is $R=\Mdd$
giving rise to NC matrix polynomials.
If $d'=d$, i.e., $R=\Rdd$, then $R\ax$ is an algebra, 
and admits an involution
fixing $\{x\}$ pointwise and being the usual transposition
on $\Rdd$.
We also use $*$ to denote the canonical mapping
$\R^{d'\times d}\ax\to\R^{d\times d'}\ax$.

A {\bf matrix NC polynomial} is an NC polynomial with matrix coefficients, i.e., an element of $\Mdd\ax$ for some $d',d\in\NN$.

\subsubsection{Polynomial evaluations}

If $p\in\Mdd\ax$ is an NC polynomial and $X\in\left(\Mnns\right)^{g}$,
the evaluation $p(X)\in\R^{d'n\times dn}$ is defined by simply replacing $x_{i}$ by $X_{i}$.
For example, if $p(x)=A x_{1} x_{2}$, where
$$
A=\begin{bmatrix}
-4&2\\
3&0
\end{bmatrix},
$$
then
$$
p\left( \begin{bmatrix}
0&1 \\
1 & 0
\end{bmatrix},
\begin{bmatrix}
1&0\\
0&-1
\end{bmatrix}
\right)=
A\otimes \left( \begin{bmatrix}
0&1 \\
1 & 0
\end{bmatrix}
\,
\begin{bmatrix}
1&0\\
0&-1
\end{bmatrix}
\right)=
\begin{bmatrix}
 0 & 4 &  0 & -2 \\
 -4 & 0 &  2 & 0 \\
 0 & -3 & 0 & 0 \\
 3 & 0 &  0 & 0 \\
\end{bmatrix}.
$$
Similarly, if $p(x)=A$ and $X\in\left(\Mnns\right)^{g}$,
then $p(X)=A\otimes I_n$.

Most of our evaluations will be on tuples of \emph{symmetric}
matrices $X\in\left(\SRnn\right)^{g}$; our involution
fixes the variables $\x$ element-wise, so only these evaluations
give rise to $*$-representations of NC polynomials.

\subsection{Archimedean quadratic modules and a Positivstellensatz}

In this subsection we use the linear Positivstellensatz
(Corollary \ref{cor:tau})
to prove
that linear pencils with bounded LMI sets give rise to archimedean
quadratic modules. This is then used to prove
a (nonlinear) Positivstellensatz for matrix
NC polynomials positive (semi)definite on bounded 
matricial LMI sets.

\begin{thm}\label{thm:pos}
Suppose $L\in\SRdd\ax$ is a monic linear pencil
and $\cD_\cL$ is bounded. Then for every symmetric polynomial
$f\in\R^{\ell\times\ell}\langle x\rangle$ with
$f|_{\cD_\cL}\succ0$, there are
$A_j\in\R^{\ell\times\ell}\langle x\rangle$,
and $B_k\in\R^{d\times\ell}\langle x\rangle$
satisfying
\beq\label{eq:pos1}
f=\sum_j A_j^*A_j + \sum_k B_k^* \cL B_k.
\eeq
\end{thm}

\begin{cor}\label{cor:pos}
Keep the assumptions of Theorem {\rm\ref{thm:pos}}.
Then
for a symmetric polynomial
$f\in\R^{\ell\times\ell}\langle x\rangle$
the following are equivalent:
\ben[\rm (i)]
\item
$f|_{\cD_\cL}\succeq0$;
\item
for every $\ep>0$ there are
$A_j\in\R^{\ell\times\ell}\langle x\rangle$,
and $B_k\in\R^{d\times\ell}\langle x\rangle$
satisfying
\beq\label{eq:pos2}
f +\eps=\sum_j A_j^*A_j + \sum_k B_k^* \cL B_k.
\eeq
\een
\end{cor}

\begin{proof}
Obviously, (ii) $\Rightarrow$ (i). Conversely, if (i) holds,
then $f+\ep|_{\cD_\cL}\succ0$ and we can apply Theorem \ref{thm:pos}.
\end{proof}

We emphasize that convexity of $\cD_L$ implies 
concrete bounds on the size of the matrices $X\in\cD_L$ 
that need to be plugged into $f$ to check whether $f|_{\cD_\cL}\succ0$:

\begin{prop}[cf.~\protect{\cite[Proposition 2.3]{HM}}]\label{prop:23}
Let $L\in\SRdd\ax$ be a linear pencil with $\cD_L$ bounded, and let
$f=f^*\in\Mnns\ax$ be of degree $m$. 
Set $s:=n \sum_{j=0}^m g^j$.
Then:
\ben[\rm (1)]
\item
$f|_{\cD_L} \succ0$ if and only if $f|_{\cD_L(s)} \succ0$;
\item
$f|_{\cD_L} \succeq0$ if and only if $f|_{\cD_L(s)} \succeq0$.
\een 
\end{prop}

\begin{proof}
In both statements the direction $(\Rightarrow)$ is obvious.
If $f|_{\cD_L} \not\succ0$, there is an $\ell$, $X\in\cD_L(\ell)$ and
$v=\oplus_{j=1}^n v_j \in (\R^\ell)^n$ with 
$$\langle f(X)v,v\rangle\leq0.$$

Let 
$$
\cK:=\{ w(X) v_j \mid w\in\ax \text{ is of degree}\leq m,\; j=1,\ldots,n\}.
$$
Clearly, $\dim \cK \leq n \sum_{j=0}^m g^j=s$.
Let $P$ be the orthogonal projection of $\R^{\ell}$ onto $\cK$. Then
$$
\langle f(PXP)v,v \rangle = \langle f(X)v,v \rangle \leq0.
$$
Since $PXP\in\cD_L(s)$, this proves (1). The proof of (2) is
the same.
\end{proof}

The crucial step in proving Theorem \ref{thm:pos}
is observing that the quadratic module generated by $\cL$ in
$ \R^{\ell\times\ell}\langle x\rangle$ is archimedean.
This is essentially a consequence of
Corollary \ref{cor:tau}, i.e., of the linear Positivstellensatz
as we now demonstrate.

\begin{definition}
Let $\cA$ be a ring with involution $a\mapsto a^*$
and set $\sym\cA:=\{a\in\cA\mid a=a^*\}$.
A subset $M\subseteq\sym\cA$ is called a {\bf quadratic module} in $\cA$
if
$$1\in M,\quad M+M\subseteq M\quad\text{and}\quad a^*Ma\subseteq M\text{\ for all $a\in\cA$.}$$
\end{definition}

\newcommand{\axstar}{\langle x\rangle}
\newcommand{\fastar}[1]{\R^{ {#1}\times{#1} }\langle x\rangle}
We will be mostly interested in the case $\cA=\fastar{\ell}$.
In this case given a subset $S\subseteq \sym\fastar d$, the quadratic
module $M_S^\ell$ generated by $S$ in $\fastar{\ell}$ is the
smallest subset of $\sym\fastar {\ell}$ containing 
all $a^* s a$ for $s\in S\cup\{1\}$, $a\in \R^{d\times\ell}\axstar$,
and closed under addition:
$$
M_S^\ell=
\Big\{ \sum_{i=1}^N a_i^*s_i a_i\mid N\in\NN,\, s_i\in S\cup \{1\},\, a_i\in\R^{d\times\ell}\ax\Big\}.
$$
This notion extends naturally to quadratic modules generated by
$S\subseteq\bigcup_{d\in\NN}\sym\fastar d$.

\begin{definition}
A quadratic module $M$ of a ring with involution $\cA$ is 
{\bf archimedean} if
\begin{equation}\label{eq:arch}
\forall a\in\cA \;\exists N\in\NN:\; N-a^*a \in M.
\end{equation}
To a quadratic module $M\subseteq\sym\cA$ we associate its
{\bf ring of bounded elements}
$$
H_M(\mathcal A):=\{a\in\cA\mid  \exists N\in\NN:\; N-a^*a \in M\}.
$$
A quadratic module $M\subseteq\sym\cA$ is thus archimedean if
and only if $H_M(\cA)=\cA$.
\end{definition}

The name \emph{ring} of bounded elements
is justified by the following proposition:

\begin{prop}[Vidav \cite{Vid}]\label{prop:vidav}
Let $\cA$ be an $\R$-algebra with involution, and
$M\subseteq\sym\cA$ a quadratic module. Then $H_M(\cA)$ is a subalgebra
of $\cA$ and is closed under the involution.
\end{prop}

Hence it suffices to check
the archimedean condition \eqref{eq:arch} on a set of algebra generators.

\begin{lemma}\label{lem:arch}
A quadratic module $M\subseteq\fastar{\ell}$ is archimedean if and only if
there exists $N\in\NN$ with $N-x^*x=N-\sum_ix_i^2\in M$.
\end{lemma}

\begin{proof}
The ``only if'' direction is obvious. For the converse, observe
that $\fastar{\ell}$ is generated as an $\R$-algebra by $\x$ and
the $\ell\times\ell$ matrix units $E_{ij}$, $i,j=1,\ldots,\ell$. By assumption,
$$
N-x_i^2= (N-\sum_ix_i^2 ) + \sum_{j\neq i}x_j^2 \in M,
$$
so $x_j\in H_M(\fastar{\ell})$ for every $j$.
On the other hand, $E_{ij}^*E_{ij}= E_{jj}$ and
thus
$$1-E_{ij}^*E_{ij}=\sum_{k\neq j} E_{kk}^*E_{kk} \in M.$$
Hence by Proposition \ref{prop:vidav}, $H_M(\fastar{\ell})=
\fastar{\ell}$ so $M$ is archimedean.
\end{proof}

We are now in a position to give our crucial observation.

\begin{prop}\label{prop:linArch}
Suppose $L\in\mbS\R^{d\times d}\ax$ is a monic linear pencil
and $\cD_\cL$ is bounded.
Then the quadratic module $M_{ \{\cL\}}^\ell$ generated
by $\cL$ in $\fastar{\ell}$ is archimedean.
\end{prop}

To make the proof more streamlined we separate one easy
argument into a lemma:

\begin{lem}\label{lem:linArch}
For $S\subseteq\bigcup_{d\in\NN}\sym\fastar d$
the following are equivalent:
\ben[\rm (i)]
\item
$M_S^\ell$ is archimedean for some $\ell\in\NN$;
\item
$M_S^\ell$ is archimedean for all $\ell\in\NN$.
\een
\end{lem}

\begin{proof}
(ii) $\Rightarrow$ (i) is obvious. For the converse assume
(i) and let $p\in\NN$ be arbitrary.
By assumption, there is $N\in\NN$ with $(N-x^*x) I_\ell \in M_S^\ell$.
If $E_{ij}^{(s,q)}$ denote the $s\times q$ matrix units, then
$$
(N-x^*x) E_{11}^{(p,p)}= (E_{11}^{(\ell,p)})^* (N-x^*x) I_\ell E_{11}^{(\ell,p)}
\in M_S^p.
$$
Now using permutation matrices we see $(N-x^*x) E_{jj}^{(p,p)}\in M_S^p$
for all $j$ concluding the proof by the additivity of $M_S^p$.
\end{proof}

\begin{proof}[Proof of Proposition {\rm \ref{prop:linArch}}]
Since $\cD_\cL$ is bounded, there is $N\in\NN$ with
$N\geq \|X\|$ for all $X\in\cD_\cL$. 
Consider the
$(g+1)\times(g+1)$ monic linear 
pencil
$$
\cJ_N(x)=
\frac 1N
\begin{bmatrix}
N & x^* \\
x & N I_g
\end{bmatrix}\in\mbS\fastar{(g+1)}
.
$$
By taking Schur complements, we see $\cJ_N(X)\succeq0$ if and only if
$N- \frac 1N \sum_j X_j^2\geq 0$, i.e., if and only
if $\|X\|\leq N$. This means $\cJ_N|_{\cD_\cL}\succeq0$ and
so by Corollary \ref{cor:tau} (in the new terminology),
$\cJ_N\in M_{\{\cL\}}^{g+1}$.
Since $M_{\{\cL\}}^{g+1}$ is closed under $*$-conjugation, we obtain
$$
\begin{bmatrix} N & 0 \\ 0 & \big(N- \frac 1N x^* x\big) I_g \end{bmatrix}
 = \begin{bmatrix} 1 & 0 \\ -\frac 1Nx & 1\end{bmatrix} N \cJ_N(x)
\begin{bmatrix} 1 & 0 \\ -\frac 1Nx & 1\end{bmatrix}^* \in
M_{\{\cL\}}^{g+1}.
$$
Again, using permutation matrices leads to
$(N^2 g - x^*x )I_{g+1}\in M_{\{\cL\}}^{g+1}$.
By Proposition \ref{prop:vidav}, $M_{\{\cL\}}^{g+1}$ is archimedean.
Finally, Lemma \ref{lem:linArch} implies $M_{\{\cL\}}^{\ell}$
is archimedean.
\end{proof}

\begin{cor}\label{cor:arch}
For a monic linear pencil $L$ the following are equivalent:
\ben[\rm (i)]
\item
$\cD_\cL(1)$ is bounded;
\item
the quadratic module $M_{ \{\cL\}}^\ell$
is archimedean for some $\ell\in\NN$;
\item
the quadratic module $M_{ \{\cL\}}^\ell$
is archimedean for all $\ell\in\NN$.
\een
\end{cor}

\begin{proof}
Clearly, (iii) $\Rightarrow$ (ii) $\Rightarrow$ (i). On the
other hand, (i) is equivalent to $\cD_\cL$ being bounded by Proposition
\ref{prop:boundedn}, so Proposition \ref{prop:linArch} applies and
allows us to deduce (iii).
\end{proof}

\begin{proof}[Proof of Theorem {\rm\ref{thm:pos}}; compare {\protect\cite[Proposition 4.1]{HM}}]
The statement \eqref{eq:pos1} holds if and only if
$f\in M_{\{\cL\}}^{\ell}$. Now that the archimedeanity
of the quadratic module $M_{\{\cL\}}^{\ell}$ has been
established in Proposition \ref{prop:linArch},
the proof is classical. We only list basic steps
and refer the reader to \cite{HM} for detailed proofs.

The proof is by contradiction, so assume $f\not\in M_{\{\cL\}}^{\ell}$.
Archimedeanity of $M_{\{\cL\}}^{\ell}$ is equivalent to the
existence of an order unit (also called algebraic interior point), namely
$1$, of the convex cone $M_{\{\cL\}}^{\ell}\subseteq \sym\fastar{\ell}$.
Thus the Eidelheit-Kakutani separation theorem yields a
linear map $\varphi:\fastar{\ell}\to\R$ satisfying
$$\varphi(f)\leq0 \quad \text{and}\quad
\varphi(M_{\{\cL\}}^{\ell})\subseteq\RR_{\geq0}.$$

\newcommand{\fastarO}{\R^{1\times\ell}\axstar}
Modding out $\cN:=\{f\in \fastarO\mid \varphi(p^*p)=0\}$ out of $\fastarO$ leads to a vector space $\cH_0$ and $\varphi$
induces a scalar product
$$
\langle \text{\textvisiblespace\,,\textvisiblespace}\rangle:\cH_0\times \cH_0\to\R, \quad (\bar p,\bar q)\mapsto
\varphi(q^*p).
$$

Completing $\cH_0$ with respect to this scalar product yields a
Hilbert
space $\cH$. It is nonzero since $\sum_i \langle e_i,e_i\rangle= \varphi(1)=1$,
where $e_i$ are the matrix units of $\R^{1\times\ell}$.
Let $e=\oplus e_i\in\cH^\ell$.

The induced left regular $*$-representation $\pi:\R\axstar\to
B(\cH)$ is bounded (since $M_{\{\cL\}}^{\ell}$ is archimedean).
Let $\hat X_i:=\pi(X_i)$ and $\hat X:=(\hat X_1,\ldots,\hat X_g)$.
The constructed scalar product extends naturally to $\cH^\ell$.
For every $\bar p\in \cH_0^\ell$, we have
$$
\langle  \cL(\hat X)  \bar p,\bar p \rangle
= \sum_{j,k} \langle \cL(\hat X)_{j,k} \bar p_j,\bar p_k\rangle
= \sum_{j,k} \varphi( p_k^* \cL(x)_{j,k}  p_j)=
\varphi( p^* \cL(x)  p)
\geq 0,
$$
where $p$ has been identified with a $\ell\times \ell$ matrix polynomial
and the last inequality results from $p^*\cL(x)p\in M_{\{\cL\}}^{\ell}$.
Hence $\hat X\in\cD_\cL$.
But now
$$
0\geq \varphi(f)= \langle f(\hat X) e, e\rangle >0,
$$
a contradiction.
\end{proof}

The cautious reader will have noticed that the constructed $\hat X$
leading to the contradiction was (in general) not acting on a
finite dimensional Hilbert space. However this is only a slight
technical difficulty; we refer the reader to 
Proposition \ref{prop:23} or
\cite[Proposition 2.3]{HM}
for a remedy.

\subsection{More constraints}

Additional constraints can be imposed on elements of a matricial LMI set.
Given $S\subseteq\bigcup_{d\in\N} \sym\Rdd\ax$, define
$$
\cD_S(n)=
\{X\in (\SRnn)^{\tg} 
\mid \forall s\in S:\; s(X)\succeq0\},$$
and let
$$
\cD_S=
 \bigcup_{n\in\NN} \cD_S(n)
$$
denote the (matrix) {\bf positivity domain}. Also of interest
is the {\bf operator positivity domain} 
$$
\cD_S^\infty =  
\big\{X\in \sym\fB(\cH)^{\tg} \mid \forall s\in S:\; s(X) \succeq 0\big\}.
$$
Here $\cH$ is a separable Hilbert space, and $\sym\fB(\cH)$
is the set of all  bounded symmetric operators on $\cH$.

\begin{thm}\label{thm:pos2}
Suppose $L\in\SRdd\ax$ is a monic linear pencil
and $\cD_\cL$ is bounded. 
Let $g_j\in\sym\R^{d_j\times d_j}\ax$ $(j\in\N)$
be symmetric matrix polynomials.
Then for every 
$f\in\sym\R^{\ell\times\ell}\langle x\rangle$ with
$f|_{\cD^\infty_{ \{L,\, g_j \mid j\in\N\}}}\succ0$, we have
$f\in M^\ell_{ \{L,\,g_j \mid j\in\N\}}.$
\end{thm}

\begin{proof}
Since the quadratic module $ M^\ell_{ \{L,\,g_j \mid j\in\N\}}\supseteq
 M^\ell_{ \{L\}}$ is archimedean, 
the same proof as for Theorem \ref{thm:pos} applies.
\end{proof}

\begin{remark}
For a particularly appealing consequence (in commuting variables)
of Theorem \ref{thm:pos2} see Section \ref{sec:comm}.
\end{remark}

We conclude this section with a Nichtnegativstellensatz.
It is a stronger form of the 
Nirgendsnegativsemidefinitheitsstellensatz
\cite{KSold} for matricial LMI sets.

\begin{cor}\label{cor:nnsd}
Let $L\in\mbS\R^{d\times d}\ax$ be a monic linear pencil
and suppose $\cD_\cL$ is bounded. 
Let $g_j\in\sym\R^{d_j\times d_j}\ax$ $(j\in\N)$
be symmetric matrix polynomials.
Then for every 
$h\in\sym\R^{\ell\times\ell}\langle x\rangle$
the following are equivalent:
\ben[\rm (i)]
\item
$h|_{\cD^\infty_{ \{L,\, g_j \mid j\in\N\}}}\not\preceq 0$, i.e., for every $($nontrivial separable$)$ 
Hilbert space $\cH$
and tuple of symmetric bounded operators $X\in\cD^\infty_{ \{L,\, g_j \mid j\in\N\}}$ on $\cH$,
there is a $v\in\cH$ with $\langle h(X)v,v\rangle>0$;
\item
there are
$D_j\in\R^{\ell\times\ell}\langle x\rangle$
satisfying
\beq\label{eq:pos4}
\sum D_j^* h D_j \in I_\ell +
M^\ell_{ \{L,\, g_j \mid j\in\N\}}.
\eeq
\een
\end{cor}

\begin{proof}
(ii) $\Rightarrow$ (i) is obvious. The converse is also
easy. 
Just apply Theorem \ref{thm:pos2} with
$f=-1$ and the positivity domain
$\cD^\infty_{ \{L,\,-h,\, g_j \mid j\in\N\}}=\emptyset.$
\end{proof}

\begin{remark}
There does not seem to exist a clean linear Nichtnegativstellensatz.
We found $4\times 4$ monic linear pencils
$L_1$, $L_2$
in nine variables with the following properties:
\ben[\rm (1)]
\item
$\cD_{L_1}$ and $\cD_{L_2}$ are bounded;
\item
$L_2|_{\cD_{L_1}}\not\preceq0$, or equivalently,
$$\left\{X\in\R^9\;\Big\vert\;\begin{bmatrix} L_1(X) & 0 \\ 0 & -L_2(X)\end{bmatrix}\succeq0\right\}=\emptyset;
$$
\item
there do not exist real matrices $U_j,V_k,W_\ell$ with
\beq\label{eq:linnsd}
\sum_j U_j^* L_2(x) U_j = I+ \sum_\ell W_\ell^* W_\ell+\sum_k V_k^* L_1(x) V_k.
\eeq
\een
By Corollary \eqref{cor:nnsd}, (1) and (2) imply that \eqref{eq:linnsd}
holds with $U_j,V_k,W_\ell\in\R^{4\times 4}\ax$.
A
Mathematica notebook with all the calculations 
is available at \url{http://srag.fmf.uni-lj.si}. 
\end{remark}

\section{More general Positivstellens\"atze}
\label{sec:morePosSS}

In this section we present two possible modifications
of our theory. First, we apply our techniques to 
commuting variables 
and derive a ``clean'' classical Putinar Positivstellensatz
on a bounded spectrahedron. This is done by adding symmetrized commutation
relations to our list of constraints. In fact we can add any symmetric
relation and get a clean Positivstellensatz on a subset of a bounded 
LMI set (this is Theorem \ref{thm:pos3}).
In Section \ref{sec:free} we also show how to deduce similar results
for nonsymmetric noncommuting variables.

\subsection{Positivstellens\"atze on an LMI set in $\R^g$}\label{sec:comm}

We adapt some of our previous definitions to commuting variables.
Let $\cy$ be the monoid freely
generated by $\y=(y_1,\ldots, y_g)$, i.e., $\cy$ consists of  words in the $g$
commuting
letters $y_{1},\ldots,y_{g}$
(including the empty word $\emptyset$ which plays the role of the identity $1$).
Let $\R\cy$ denote the commutative 
$\R$-algebra freely generated by $\y$, i.e., the elements of $\R\cy$
are polynomials in the commuting variables $\y$ with coefficients
in $\R$. 

More generally, 
for an abelian group $R$ we use $R\cy$
  to denote the abelian
  group of all $R$-linear combinations of 
words in $\cy$.
Besides $R=\R$, the most important example is $R=\Mdd$
giving rise to matrix polynomials.
If $d'=d$, i.e., $R=\Rdd$, then $R\cy$ is an $\R$-algebra, and
admits an involution
fixing $\{y\}$ pointwise and being the usual transposition
on $\Rdd$.
We also use $*$ to denote the canonical mapping
$\R^{d'\times d}\cy\to\R^{d\times d'}\cy$.
If $p\in\Mdd\cy$ is a polynomial and $Y\in\R^{g}$,
the evaluation $p(Y)\in\Mdd$ is defined by simply replacing $y_{i}$ by $Y_{i}$.

The natural map $\ax\to\cy$ is called the {\bf commutative collapse}.
It extends naturally to matrix polynomials.

   For $A_0,A_1,\dots,A_{g} \in \SRdd$,  a linear matrix polynomial
\beq\label{eq:cpencil}
  L(y)=A_0+\sum_{j=1}^{g} A_j y_j \in \SRdd\cy,
\eeq
  is a linear pencil. 
If $A_0=I$, then $L$ is monic.
If $A_0=0$, then $L$ is a truly linear pencil.
Its spectrahedron is
$$
\cD_{L}(1)=\{Y\in \R^{\tg} \mid L(Y) \succeq 0\},
$$
and for every $\ell\in\NN$, $L$ induces a quadratic module $Q_{\{L\}}^\ell$
in $\R^{\ell\times\ell}\cy$:
$$
Q_{\{L\}}^\ell =
\Big\{ 
\sum_{i=1}^N a_i^*a_i +
\sum_{j=1}^N b_j^*L b_j\mid N\in\NN,\, a_i\in \R^{\ell\times\ell}\cy, \,
b_j\in\R^{d\times\ell}\cy\Big\}.
$$

All the results on linear pencils and archimedeanity given in
Section \ref{sec:pos} carry over to
the commutative setting. For instance, given a monic linear
pencil $L\in\mbS\R^{d\times d}\cy$, we have:
\begin{enumerate}[\rm (1)]
\item
$Q_{ \{\cL\}}^\ell$
is archimedean for some $\ell\in\NN$ if and only if
$Q_{ \{\cL\}}^\ell$
is archimedean for all $\ell\in\NN$;
\item
$Q_{ \{\cL\}}^\ell$
is archimedean  if and only if
the spectrahedron $\cD_L(1)$ is bounded.
\end{enumerate}

Most importantly, we obtain the following clean version of
Putinar's Positivstellensatz \cite{Put} on a bounded spectrahedron.

\begin{thm}\label{thm:commpos}
Suppose $L\in\SRdd\cy$ is a monic linear pencil
and $\cD_\cL(1)$ is bounded. Then for every symmetric polynomial
$f\in\R^{\ell\times\ell}\cy$ with
$f|_{\cD_\cL(1)}\succ0$, there are
$A_j\in\R^{\ell\times\ell}\cy$,
and $B_k\in\R^{d\times\ell}\cy$
satisfying
\beq\label{eq:cpos1}
f=\sum_j A_j^*A_j + \sum_k B_k^* \cL B_k.
\eeq
\end{thm}

\begin{proof}
Let $F\in\sym\R^{\ell\times\ell}\ax$ be an arbitrary symmetric 
matrix polynomial 
in noncommuting variables 
whose commutative
collapse is $f$.
By abuse of notation, let $L\in\SRdd\ax$ be the canonical lift of 
$L\in\SRdd\cy$. Write 
$$
g_{ij}= -(x_ix_j-x_jx_i)^*(x_ix_j-x_jx_i)= (x_ix_j-x_jx_i)^2 \in \sym\R\ax
$$
for $i,j=1,\ldots,g$.
Note $g_{ij}(X)\succeq0$ if and only if $X_iX_j=X_jX_i$.
By the spectral theorem, $F|_{\cD_{\{L,\, g_{ij}\mid i,j=1,\ldots,g\}}}\succ0$.
So Theorem \ref{thm:pos2} implies and yields
$$
F\in 
M^\ell_{ \{L,\, g_{ij} \mid i,j=1,\ldots,g\}}.
$$
Applying the commutative collapse gives
$f\in Q^\ell_{\{L\}}$, as desired.
\end{proof}

\begin{cor}\label{cor:commpos4}
Suppose $L\in\SRdd\cy$ is a monic linear pencil
and $\cD_\cL(1)$ is bounded.
Let $g_1,\ldots,g_s\in\R\cy$ and
$$
\cD_L(g_1,\ldots,g_s):=\{Y\in\R^g\mid L(Y)\succeq0,\, g_1(Y)\geq0,\ldots,g_s(Y)\geq0\}.
$$
If
$f\in\R\cy$ satisfies
$f|_{\cD_\cL(g_1,\ldots,g_s)}>0$,
then there are
$h_{ij}\in\R\cy$,
and $B_k\in\R^{d\times1}\cy$
satisfying
\beq\label{eq:cpos4}
f =\sum_{j=0}^s g_j \sum_{i} h_{ij}^2+ \sum_k B_k^* \cL B_k,
\eeq
where $g_0:=1$.
\end{cor}

\begin{cor}\label{cor:commpos3}
Suppose $L\in\SRdd\cy$ is a monic linear pencil
and $\cD_\cL(1)$ is bounded. Then for every polynomial
$f\in\R\cy$ with
$f|_{\cD_\cL(1)}> 0 $, there are
$h_j\in\R\cy$,
and $B_k\in\R^{d\times1}\cy$
satisfying
\beq\label{eq:cpos3}
f=\sum_j h_j^2 + \sum_k B_k^* \cL B_k.
\eeq
\end{cor}

It is clear that a Nichtnegativstellensatz along
the lines of Corollary \ref{cor:nnsd} holds in this setting.
We leave the details to the reader.

\subsection{Free (nonsymmetric) variables}\label{sec:free}

In this section we explain how our theory adapts to
the free $*$-algebra. 
Let $\axs$ be the monoid freely
generated by $\x=(x_1,\ldots, x_g)$ and 
$\xs=(x_1^*,\ldots,x_g^*)$, i.e., $\axs$ consists of  words in the $2g$
noncommuting
letters $x_{1},\ldots,x_{g},x_1^*,\ldots,x_g^*$
(including the empty word $\emptyset$ which plays the role of the identity $1$).
Let $\C\axs$ denote the 
$\C$-algebra freely generated by $\x,\xs$, i.e., the elements of $\C\axs$
are polynomials in the noncommuting variables $\x,\xs$ with coefficients
in $\C$. 
As before, we introduce matrix polynomials $\C^{d'\times d}\axs$.
If $p\in\CMdd\axs$ is a polynomial and $X\in(\Cnn)^{g}$,
the evaluation $p(X,X^*)\in\CMdd$ is defined by simply replacing $x_{i}$ by $X_{i}$ and $x_i^*$ by $X_i^*$.

   For $A_1,\dots,A_{g} \in \Cdd$,  a linear matrix polynomial
\beq\label{eq:cpencilC}
  L(x)=\sum_{j=1}^{g} A_j x_j \in \Cdd\ax,
\eeq
  is a truly linear pencil. (Note: none of the variables
$\xs$ appears in such an $L$.) 
Its {\bf monic symmetric pencil} is 
$$
\ccL(x,x^*)= I + L(x) + L(x)^* = I + \sum_{j=1}^{g} A_j x_j +
 \sum_{j=1}^{g} A_j^* x_j^*
\in\sym\Cdd\axs.
$$
The associated matricial LMI set is
$$
\cD_{\ccL}=\bigcup_{n\in\N} \{X\in (\Cnn)^{\tg} \mid \ccL(X) \succeq 0\},
$$
its operator-theoretic counterpart is
$$
\cD_{\ccL}^\infty= \{X\in \fB(\cH)^{\tg} \mid \ccL(X) \succeq 0\},
$$
and for every $\ell\in\NN$, $\ccL$ induces a 
quadratic module $M_{\{\ccL\}}^\ell$
in $\C^{\ell\times\ell}\axs$:
$$
M_{\{\ccL\}}^\ell =
\Big\{ 
\sum_{i=1}^N a_i^*a_i +
\sum_{j=1}^N b_j^*\ccL b_j\mid N\in\NN,\, a_i\in \C^{\ell\times\ell}\axs, \,
b_j\in\C^{d\times\ell}\axs\Big\}.
$$

Like in the previous subsection,
all our main results from 
Section \ref{sec:pos} carry over to
this free setting. 
As a sample, we give a Positivstellensatz:

\begin{thm}\label{thm:pos3}
Suppose $\ccL\in\sym\Cdd\axs$ is a monic symmetric linear pencil
and $\cD_\ccL$ is bounded. 
Let $g_j\in\sym\C^{d_j\times d_j}\axs$ $(j\in\N)$
be symmetric matrix polynomials.
Then for every 
$f\in\sym\C^{\ell\times\ell}\axs$ with
$f|_{\cD^\infty_{ \{\ccL,\,g_j \mid j\in\N\}}}\succ0$, we have
$f\in M^\ell_{ \{\ccL,\,g_j \mid j\in\N\}}.$
\end{thm}

As a special case we obtain a Positivstellensatz describing
polynomials positive definite on \emph{commuting} tuples $X$ in a matricial LMI set. (Note: we are not assuming the entries $X_i$ commute with
the adjoints $X_j^*$.)

\begin{cor}\label{cor:pos3}
Suppose $\ccL\in\sym\Cdd\axs$ is a monic symmetric linear pencil
and $\cD_\ccL$ is bounded. 
Suppose $f\in\sym\C^{\ell\times\ell}\axs$ satisfies
$f(X,X^*)\succ0$ for all $X\in\cD^\infty_\ccL$ with
$X_iX_j=X_jX_i$ for all $i,j$.
\ben[\rm (1)]
\item
Let $c_{jk}=x_jx_k-x_kx_j$. 
Then
$$f\in M^\ell_{ \{\ccL, \,c_{jk}+c_{jk}^*,\, i (c_{jk}-c_{kj}) \mid j,k=1,\ldots,g\}}.$$
\item
Let $d_{jk}=-c_{jk}^*c_{jk}$. Then
$$f\in M^\ell_{ \{\ccL,\, d_{jk} \mid j,k=1,\ldots,g\}}.$$
\een
\end{cor}

\linespread{1.1}

\end{document}